\documentclass[lettersize,onecolumn]{IEEEtran}
\usepackage{amsmath,amsfonts,amsthm}
\usepackage{algorithmic}
\usepackage{algorithm}
\usepackage{array}
\usepackage[caption=false,font=normalsize,labelfont=sf,textfont=sf]{subfig}
\usepackage{textcomp}
\usepackage{stfloats}
\usepackage{url}
\usepackage{verbatim}
\usepackage{graphicx}
\usepackage{cite}
\usepackage{hyperref}
\usepackage{mathrsfs,nicefrac}
\usepackage[shortlabels]{enumitem}
\usepackage{xcolor}
\usepackage{multirow} 

\DeclareMathOperator*{\argmin}{argmin}
\DeclareMathOperator{\ksd}{KSD}
\DeclareMathOperator{\wksd}{wKSD}
\DeclareMathOperator{\tr}{tr}
\DeclareMathOperator{\SO}{SO}
\DeclareMathOperator{\SE}{SE}
\DeclareMathOperator{\Div}{div}
\DeclareMathOperator{\vmf}{vMF}
\DeclareMathOperator{\Log}{Log}
\DeclareMathOperator{\vectz}{vec}
\DeclareMathOperator{\mle}{MLE}

\newtheorem{theorem}{Theorem}%[section]
\newtheorem*{theorem*}{Theorem}
\newtheorem{corollary}{Corollary}[theorem]
\newtheorem{lemma}{Lemma}%[section]
\newtheorem{example}{Example}%[section]
\newtheorem{proposition}{Proposition}%[section]

\begin{document}

\title{Kernel Stein Discrepancy on Lie Groups:\\Theory and Applications}

\author{Xiaoda Qu, Xiran Fan, Baba C. Vemuri \thanks{Manuscript received February 2024}}

%\author{Xiaoda Qu, Xiran Fan, Baba C. Vemuri,~\IEEEmembership{Fellow,~IEEE,}
        % <-this % stops a space
%\thanks{Manuscript received February 2024}}

% The paper headers
% \markboth{IEEE Trans. on Info. Theorys,~Vol.~XX, No.~XX, 2024}%
% {Shell \MakeLowercase{\textit{et al.}}: A Sample Article Using IEEEtran.cls for IEEE Journals}

%\IEEEpubid{0000--0000~\copyright~2023 IEEE}
% Remember, if you use this you must call \IEEEpubidadjcol in the second
% column for its text to clear the IEEEpubid mark.

\maketitle

\begin{abstract}
Distributional approximation is a fundamental problem in machine learning with numerous applications across all fields of science and engineering and beyond. The key challenge in most approximation methods is the need to tackle the intractable normalization constant present in the candidate distributions used to model the data. This intractability is especially common for distributions of manifold-valued random variables such as rotation matrices, orthogonal matrices etc. In this paper, we 
focus on the distributional approximation problem in Lie groups since they are frequently encountered in many applications including but not limited to, computer vision, robotics, medical imaging and many more. We
present a novel Stein's operator on Lie groups leading to a kernel Stein discrepancy (KSD) which is a normalization-free loss function. We present several theoretical results characterizing the properties of this new KSD on Lie groups and its minimizer namely, the minimum KSD estimator (MKSDE). Properties of MKSDE are presented and proved, including strong consistency, CLT and a closed form of the MKSDE for the von Mises-Fisher, the exponential and the Riemannian normal distributions on \(\SO(N)\). Finally, we present several experimental results depicting advantages of MKSDE over maximum likelihood estimation.
\end{abstract}

%\end{IEEEkeywords}

\section{Introduction (Problem Motivation and Literature Review)}

Distributional approximation is a ubiquitous problem in science and engineering with numerous applications. 
This fundamental problem can be formulated as follows: Suppose we wish to approximate a density \(q\) by a family of candidate probability densities  \(\{p_\alpha\}\), either parameterized by some finite dimensional parameters or non-parametrically represented via say, a neural network. The approximation is invariably accomplished via the optimization of a loss (cost) function \(D\), which measures the dissimilarity between distributions, i.e., \(\alpha^*:=\argmin D(q\Vert p_\alpha )\). A vanilla application of this fundamental technique known as maximum likelihood estimation (MLE), is an important constituent of a wide variety of algorithms, e.g., the well known expectation-maximization (EM) algorithm, the recursive stochastic filter namely the Kalman filter \cite{Sorenson_85} and many others.

Most commonly used loss functions include likelihood or log-likelihood functions and the KL-divergence. However, in many cases, the family of candidate densities \(\{p_\alpha\}\) are only known up to a highly intractable normalizing constant. In practice, one must approximate the normalizing constant and its derivative w.r.t \(\alpha\) numerically in each step of say a gradient decent method employed in the optimization.  {\it An intractable normalizing constant leads to a cost-accuracy trade-off i.e., higher the computational cost, better the accuracy of approximation. If one can circumvent this intractable constant altogether and yet achieve high accuracy in parameter estimation, it would be ideal and this is precisely what we will achieve in this paper.}

Such a situation arises even in high dimensional Euclidean space, not to mention the non-flat spaces. In many computer vision and machine learning applications such as, unmanned aerial vehicle tracking, aircraft trajectory analysis, human motion analysis etc., one encounters data in different Lie groups, as it is the most appropriate space to represent different types of object motions e.g., rotations, affine motions etc. Lie groups 
capture the underlying intrinsic geometric structure of these transformations that represent the object motions. In contrast, a vector space structure proves to be inadequate for most tasks in manipulating such data. For example, the arithmetic mean  of several rotational motions -- rotation (orthogonal) matrices -- of an object is no longer a rotation. Therefore, to model the rotation of objects, it is common to use the geometric structure of the special orthogonal group \(\SO(N):=\{R\in \mathbb{R}^{N\times N}:R^\top R=I, \det(R)=1\}\) instead of \(\mathbb{R}^{N\times N}\). 

The issue of normalizing constant is significant for distributions in Lie groups. The integral on Lie groups w.r.t its volume measure, though well-defined, most often is computationally intractable. Even the most widely-used, well-behaved probability families on Lie group have an intractable normalizing constant. For example, von Mises-Fisher distribution (vMF) \cite{downs1972orientation} and Bingham distribution \cite{hoff2009simulation, khatri1977mises} on \(\SO(N)\), whose normalizing constants are hypergeometric functions of the parameters, are widely used in pose estimation or rotation estimation with uncertainty \cite{deng2022deep,gilitschenski2020deep,peretroukhin2020smooth,prokudin2018deep} and Bayesian attitude estimation \cite{lee2018bayesian,wang2021matrix} and many others. The Riemannian normal distribution \cite{pennec2006intrinsic}, with a highly intractable normalizing constant, is encountered in probabilistic principal geodesic analysis (PPGA)
\cite{zhang2013probabilistic,zhang2019mixture}. The one-axis model on special Euclidean group \(\SE(3)\)\cite[Ch. 6]{ley2018applied}, for modeling the rigid body motion \cite{oualkacha2012estimation} contains a vMF component. These are just a few examples where intractability of the normalizing constant is encountered.

A normalization-free loss, called \emph{kernel Stein discrepancy} (KSD), was first proposed in \cite{oates2017control,chwialkowski2016kernel}, which combines Stein's method with the theory of reproducing kernel Hilbert space (RKHS). The KSD has been extensively studied recently in different aspects of a general framework \cite{ley2017stein,mijoule2018stein,mijoule2023stein}, characterization scope \cite{gorham2015measuring,gorham2017measuring}, 
asymptotic properties regarding its minimization \cite{barp2019minimum,oates2022minimum}, and its relevant applications \cite{chwialkowski2016kernel,liu2016kernelized,matsubara2022robust}. It involves an RKHS \(\mathcal{H}_k\) of a kernel function \(k\) on \(\mathbb{R}^d\), termed as \emph{Stein's class}  here, and a \emph{Stein operator} \(S_{p_\alpha}\). The Stein operator \(\mathcal{S}_{p_\alpha}\), depends on the candidate distribution \(p_\alpha\) but is free from its normalizing constant, maps the elements in \(\mathcal{H}^d_k\), the \(d\)-times product of \(\mathcal{H}_k\), to real-valued functions. In addition, they must satisfy \emph{Stein's identity}, that is, \(\mathbb{E}_p[\mathcal{S}_p f]=0\) for all \(f\in \mathcal{H}^d_k\). Then, the KSD between \(p\) and another distribution \(q\) is defined as,
\begin{equation}\label{KSDdef}
    \ksd(p,q):= \sup\big\{  \mathbb{E}_q[\mathcal{S}_p f]: f\in\mathcal{H}^d_k, \Vert f\Vert_{\mathcal{H}^d_k}\leq 1 \big\}.
\end{equation}

\emph{In contrast to the KSD on \(\mathbb{R}^d\), the existing generalizations of KSD to non-flat manifolds are somewhat lacking.} 
In \cite{barp2018riemann} Barp et al. developed a novel KSD limited to compact Riemannian manifolds and studied its convergence property using a higher order Sobolev space as the RKHS. %in their work.
 Recently, in \cite{xu2021interpretable} Xu and Matsuda implemented the Euclidean KSD in a local coordinate chart, thus the KSD in their work is only applicable to distributions supported within the local coordinate chart. A more detailed discussion on these two works will be presented in section \ref{KSDE}.

\subsection*{Our Contributions and Paper Organization} 
In this work, we propose a novel Stein's operator (see equation \ref{StOG}) making use of the structure of Lie groups, leading to a KSD with a closed form (see equation \ref{KSDForm}), {\it a normalization-free loss function applicable to all Lie groups}. The asymptotic properties of the KSD and its minimizer are established in theorems \ref{KSDasymptotic}, \ref{KSDasymptotic-stronger}, \ref{MKSDEconsistency} and \ref{MKSDECLT} respectively. Furthermore, we present applications of the KSD to important problems on the widely encountered Lie group \(\SO(N)\) in \S \ref{ExpSO(N)} and example \ref{vMF-SO(N)}. Specifically, experiments \ref{MKSDEvsMLE} will address the issue of the normalization constant that arises in MLE and how the estimation obtained using proposed normalization-free KSD yields far more accurate parameter estimates compared to MLE that uses approximations for the normalization constant.

\paragraph{Organization:} The rest of this paper is organized as follows: In Section \ref{background}, we present the mathematical preliminaries required in the rest of the paper. Section \ref{KSD-LG} contains the key theoretical results involving the derivation of the KSD on Lie groups. In Section \ref{MKSDE-section} we present the minimum kernel Stein discrepancy estimator and its asymptotic properties, and analyze different practical situations that arise in applications. Proofs of all theorems are included in the supplementary. Finally, we present experiments in Section \ref{experiments} and conclude in Section \ref{conclusion}. 
\section{Background} \label{background}
In this section, we present a brief overview of relevant mathematics on Lie groups and the KSD on \(\mathbb{R}^d\).
For a more detailed discussion on Lie groups, we refer the reader to \cite{helgason1979differential}.

\subsection{Manifolds}

A manifold $M$ is a second countable Hausdorff space that is locally homeomorphic to an open set in $\mathbb{R}^d$. The number $d$ is the \emph{dimension} of $M$. Such local homeomorphisms are called \emph{charts}. A smooth structure on $M$ is assigned via a collection of charts that covers the manifold such that the composition of every two charts is smooth on $\mathbb{R}^d$. A manifold with a smooth structure is called a \emph{smooth manifold}.

\subsection{Lie groups} \label{LG}

A \emph{Lie group} \(G\) is a group and a smooth \(d\)-dimensional manifold whose multiplication and group inverse are smooth. For any fixed \(g\in G\), the \emph{left (resp. right) multiplication action} \(L_g: h\mapsto g h\), (resp. \(R_g: h\mapsto h g\)) is smooth, whose differential is denoted by $d L_g$ (resp. $d R_g$ ) and its
pullback is denoted by \(L_g^*\) (resp. \(R_g^*\) ).

\paragraph{Left-invariant fields} A vector field is a smooth assignment of tangent vectors to the tangent space at each point of $G$. A \emph{left-invariant field} is a vector field invariant w.r.t left-translations, i.e., \(d L_g\) or \(L_g^*\) for any \(g\in G\). Clearly, if a group of left-invariant vector fields is linearly independent at some point, then they are linearly independent at every point. By the term \emph{left-invariant vector (resp. covector or tensor) basis}, we mean a collection of left-invariant vector (resp. covector or tensor) fields that forms a basis at every point. Given a left-invariant basis, any other left-invariant field can be written as the linear combination of the given basis.

\paragraph{Volume measure and modular functions} Akin to the Lebesgue measure on $\mathbb{R}^d$, the canonical dominating measure on a Lie group is usually taken as the \emph{left-invariant Haar measure} $\mu$, defined via left invariant volume form $\Omega$ or Haar theorem, see \cite{lee2013smooth} and \cite{evans2018measure}. The left-invariant measure $\mu$ satisfies \(\mu(g S)=\mu(S)\) for all \(g\in G\) and Borel set \(S\), which is unique up to a constant. For each fixed \(g\in G\), note that $\mu_g(S):=\mu(Sg)$ is still a left-invariant measure, thus there exists a number \(\Delta_g>0\) such that $\mu_g=\Delta(g) \mu$. The function \(\Delta: g\mapsto \Delta_g\) is the \emph{modular function } of \(G\). We say that \(G\) is \emph{unimodular}, if \(\Delta\equiv 1\), e.g., $\SO(N)$, $\SE(N)$, etc. All Abelian or compact Lie groups are unimodular. All probability densities used in this paper are with respect to the volume measure \(\mu\).

\subsection{KSD on \texorpdfstring{\(\mathbb{R}^d\)}{R\^d}
and existing generalizations} \label{KSDE}

Suppose \(\mathcal{H}_k\) is an RKHS on \(\mathbb{R}^d\) of a kernel \(k\). Let \(\mathcal{H}^d_k:=\mathcal{H}_k\times\cdots\times\mathcal{H}_k\) be the \(d\)-times product of \(\mathcal{H}_k\), endowed with the inner product
\( \langle f,g\rangle_{\mathcal{H}_k^d}=\sum_{l=1}^d\langle f_l,g_l \rangle_{\mathcal{H}_k} \) for \(f=(f_1,\dots,f_d)\) and \(g=(g_1,\dots,g_d)\) in \(\mathcal{H}^d_k\). The most commonly used Stein operator \(\mathcal{S}_p\) on \(\mathbb{R}^d\) also adopted in \cite{barp2019minimum,chwialkowski2016kernel,gorham2017measuring,liu2016kernelized,matsubara2022robust,oates2022minimum}, denoted here by $T_p$, has the form 
\begin{equation}\label{StOE}
    \mathcal{T}_p: f \mapsto \sum_{l=1}^d \left[\frac{\partial f_l}{\partial x^l} + f_l \frac{\partial}{\partial x^l} \log p\right],\quad f\in \mathcal{H}^d_k.
\end{equation}

Then the KSD is defined by 
\[
\ksd(p,q):= \sup\big\{  \mathbb{E}_q[\mathcal{T}_p f]: f\in\mathcal{H}^d_k, \Vert f\Vert_{\mathcal{H}^d_k}\leq 1 \big\}.
\]
Clearly, one could easily see from (\ref{KSDdef}) and Stein's identity that the KSD is always non-negative and satisfies \(\ksd(p,p)=0\). In fact, as discussed in \cite{barp2019minimum,chwialkowski2016kernel,gorham2017measuring,liu2016kernelized,matsubara2022robust,oates2022minimum}, if the kernel function \(k\) is \(C_0\)-universal \cite{carmeli2010vector}, then KSD will uniquely characterize \(p\), i.e., for \(p\) and \(q\) regular enough, we have \(\ksd(p,q)=0 \Leftrightarrow p=q\). Notably, the computation of \(T_p\) is free from the normalizing constant of \(p\), so is the corresponding KSD in (\ref{KSDdef}). 

As one of the significant application of KSD, the \emph{minimum kernel Stein discrepancy estimator} (MKSDE), first proposed by \cite{barp2019minimum} and further researched in \cite{matsubara2022robust,oates2022minimum}, minimizes the KSD between a parametrized family \(p_\theta\) and the empirical distribution of a group of samples to obtain an approximation \(p_{\theta^*}\) of the underlying distribution of the samples. The MKSDE has good convergent properties and thus can serve as a normalization-free alternative of MLE.

Since the issue of normalization is even more severe on non-flat spaces, one would naturally wish to generalize such a method to manifolds. In \cite{xu2021interpretable}, Xu and Matsuda use a local chart to construct a \emph{differential form} $\omega$, such that its differential $d\omega=p\mathcal{S}_p f \Omega$ ($\Omega$ is the volume form), and then apply Stokes's theorem to conclude that $\mathbb{E}_p[\mathcal{S}_p f]=\int p \mathcal{S}_p f \Omega=\int d\omega=0 $. However, as their construction relies on the local coordinates,  $p\mathcal{S}_p f\Omega$ does not extend beyond the chart and is singular on the boundary of the chart, however, Stokes's theorem does require $p\mathcal{S}_p f\Omega$ to be globally smooth. Thus, the applicability of this method is limited to cases where $p$ is supported inside the chart so that the $p\mathcal{S}_p f\Omega$ vanishes (smoothly) before reaching the boundary, or when a global chart exists. It is however well known that there is no global chart on any compact manifold, and most commonly used families are all globally supported, e.g., the vMF family on SO$(N)$. Xu and Matsuda \cite{xu2021interpretable} present an example of the vMF family on the sphere and as stated earlier, the sphere has no global chart (spherical coordinates are not global as they are singular on the boundary) and it violates their Stein's operator construction that depends on the local charts. Further, the volume density $J$ in their work is easy to compute on a sphere but not so on more complicated spaces, such as SO$(N)$ and others.

In \cite{barp2018riemann}, Barp et al. adopted a second order Stein's operator free from the choice of coordinate charts in constructing an approximation to the posterior expectation of distributions supported only on compact manifolds. In addition, to successfully characterizing \(p\) uniquely with the KSD corresponding to this operator, one must consider a very large RKHS, i.e., the Sobolev space on $M$.
Finding such a Sobolev space with an easy-to-derive closed form kernel can be challenging on a highly curved manifold. One possible way to address this challenge is to use the restriction of a Sobolev-type kernel on $\mathbb{R}^d$ to the manifold \cite{fuselier2012scattered}.

\section{KSD on Lie groups}\label{KSD-LG}
In this section, we present a novel Stein's operator on Lie groups denoted here by \(\mathcal{A}_p\) and derive its corresponding KSD, making use of the group structure, specifically, left-invariant basis, akin to the basis \(\frac{\partial}{\partial x^l}\) on \(\mathbb{R}^d\).

\subsection{Stein's operator on Lie groups and associated KSD}

Suppose \(G\) is a connected Lie group with a modular function \(\Delta\) and \(\{D^l\}_{l=1}^d\) is a left-invariant basis on \(G\). Suppose \(p\) is a locally Lipschitz probability density function on \(G\), only known up to a constant. Suppose \(\mathcal{H}_k\) is an RKHS with a kernel function \(k\) on \(G\). Then, the \emph{Stein's operator} \(\mathcal{A}_p\) is defined as,
\begin{equation}\label{StOG}
\mathcal{A}_p: f\mapsto \sum_{l=1}^d \left[ D^l f_l+ f_l D^l\log p + f_l D^l\Delta\right], \quad f\in \mathcal{H}^d_k.
\end{equation}
Here \(D^l \log p \) is set to \(0\) whenever \(p=0\). In particular, we specify each component \(\mathcal{A}^l_p\) of the operator \(\mathcal{A}_p\) as \(\mathcal{A}^l_p: h\mapsto  D^l h+ h D^l\log p + h D^l\Delta\) for \(h\in \mathcal{H}_k\).
Thus, \(\mathcal{A}_p f=\sum_{l=1}^d \mathcal{A}^l_p f_l\). In addition, the vector \((\mathcal{A}^1_p f_1,\dots,\mathcal{A}^d_p f_d )^\top\) is denoted by \(\vec{\mathcal{A}}_p f\).

On unimodular groups, e.g, compact Lie groups or abelian Lie groups, the last term \(f_l D^l\) disappears since the modular function \(\Delta\) is a constant. Specifically, note that \((\mathbb{R}^d,+)\) is an abelian Lie group, thus \(\mathcal{A}_p\) degenerates to \(T_p\) in (\ref{StOE}) on \(\mathbb{R}^d\). In contrast to the generalization by \cite{xu2021interpretable}, the computation of \(\mathcal{A}_p\), as well as its corresponding KSD defined through (\ref{KSDdef}), is independent of the choice of a coordinate chart. Although \(\mathcal{A}_p\) seemingly depends on the choice of left-invariant basis \(D^l\), we will show later that different choices of \(D^l\) will actually lead to equivalent KSDs related to each other via the inequality (\ref{KSD-Ineq}).

Following result ensures that all \(f\in\mathcal{H}_k\) are differentiable so that \(\mathcal{A}_p f\) is meaningful.

\begin{theorem}\label{KSDdiff} If \(k\in C^2(G\times G)\), i.e., twice continuously differentiable, then all \(f\in \mathcal{H}_k\) is \(C^1\). Furthermore, given a tangent vector $D\in T_{x_0} M$, $(D k)_{x_0}\in\mathcal{H}_k$ and
$D f(x_0)=\langle f,(D k)_{x_0}\rangle_{\mathcal{H}_k}$ for all $f\in\mathcal{H}_k$. Here $(D k)_{x_0}$ represents the function obtained by letting $D$ act on the first argument of $k$ and fix the first argument at $x_0$.
\end{theorem}

\begin{proof} See \cite[Lem. 4.34]{steinwart2008support}.
\end{proof}

From here on, we always assume that our kernel function is at least \(C^2\).
The \emph{kernel Stein discrepancy} (KSD) on Lie groups defined via (\ref{KSDdef}) and (\ref{StOG}), share most of properties with the vanilla KSD on \(\mathbb{R}^d\) \cite{chwialkowski2016kernel,liu2016kernelized}, defined via (\ref{KSDdef}) and (\ref{StOE})   -- specifically, replacing $S_p$ in (\ref{KSDdef}) with $T_p$ from (\ref{StOE}) in the $\mathbb{R}^d$ case and $\mathcal{A}_p$ from (\ref{StOG}) in the Lie group case respectively. For example, the shared properties include but are not limited to (i) non-negativeness, and most significantly, (ii) an integral closed form (\ref{KSDForm}). This remarkable closed form is based on the structure of the RKHS. We denote by \(\mathcal{A}^l_p k_x:=\mathcal{A}^l_p  k(x,\cdot)\) the function obtained by letting \(\mathcal{A}^l_p\) act on the first argument of \(k\) and then fix the first argument of $\mathcal{A}^l_p k(\cdot,\cdot)$ at $x$. Under the condition in theorem \ref{KSDdiff}, it is clear that \(\mathcal{A}^l_p k_x\in\mathcal{H}_k\) for all \(x\), and \(\mathcal{A}^l_p h(x)=\langle h,\mathcal{A}^l_p k_x\rangle_{\mathcal{H}_k}\) for all \(h\in\mathcal{H}_k\). Furthermore, the expectation \(\mathbb{E}_q[\mathcal{A}^l_p h]\) can be written as \(\mathbb{E}_q\langle h,\mathcal{A}_p^l k_{(\cdot)}\rangle_{\mathcal{H}_k} = \langle h, \mathbb{E}_q[\mathcal{A}^l_p k_{(\cdot)}]  \rangle_{\mathcal{H}_k} \). The expectation \(\mathbb{E}_q[\mathcal{A}^l_p k_{(\cdot)}] \) is the Bochner-integral \cite[\S A.5.4]{steinwart2008support} of the map \(x\mapsto \mathcal{A}_p^l k_x \) from \(G\) to \(\mathcal{H}_k\) w.r.t to \(q\), which is well-defined whenever the map \(x\mapsto \mathcal{A}_p^l k_x \) is Bochner \(q\)-integrable \cite[\S A.5.4]{steinwart2008support}. In such a case, the KSD can be represented by
\[
\ksd(p,q)
= \sup_{\Vert f\Vert_{\mathcal{H}^d_k}\leq 1} \mathbb{E}_q[\mathcal{A}_p f]=\sup_{\Vert f\Vert_{\mathcal{H}^d_k}\leq 1}\langle f,\mathbb{E}_q[\Vec{\mathcal{A}}_p k_{(\cdot)}]\rangle_{\mathcal{H}^d_k}=\Vert  \mathbb{E}_q[\Vec{\mathcal{A}}_p k_{(\cdot)}] \Vert_{\mathcal{H}^d_k}.
\]
Substituting the term \(\mathcal{A}^l_p k_{y}\) for \(h\) in the property \(\mathcal{A}^l_p h(x)=\langle h,\mathcal{A}^l_p k_x\rangle_{\mathcal{H}_k}\), we have,
\begin{equation*}
k_p(x,y):=\sum_{l=1}^d k^l_p(x,y):= \sum_{l=1}^d \Tilde{\mathcal{A}}^l_p\mathcal{A}^l_p  k(x,y) = \sum_{l=1}^d \langle  \mathcal{A}^l_p k_x, \mathcal{A}^l_p k_y\rangle_{\mathcal{H}_k}. 
\end{equation*}
Here, the tilde on \(\Tilde{\mathcal{A}}^l_p\) indicates that it acts on the second argument. Interchanging the inner product and the expectation again, we have the closed form of KSD, summarized in the following theorem, which is a {\it generalization to the classical version presented in} \cite{chwialkowski2016kernel,liu2016kernelized} on $\mathbb{R}^d$.

\begin{theorem}[Closed form KSD]\label{KSDFormThm} Suppose $k\in C^2(G\times G)$ and \(\sqrt{k_p(x,x)}\) is $p$ and $q$-integrable, then the map \(x\mapsto \mathcal{A}^l_p k_x\) is Bochner \(q\)-integrable and
\begin{equation}\label{KSDForm}
\ksd^2(p,q)=\int_G \int_G k_p(x,y) q(x) q(y)\mu(d x) \mu(dy). 
\end{equation}
\end{theorem}

\begin{proof} The proof is along the same lines of the theroem in the $\mathbb{R}^d$ case, see e.g. \cite{liu2016kernelized,chwialkowski2016kernel}, and hence omitted here.    
\end{proof}
With theorem \ref{KSDFormThm} in hand, we can establish the invariance of KSD with respect to different choice of basis. In particular, orthogonal transformation between basis will preserve the KSD. The following theorem formalizes this invariance.

\begin{theorem}[Invariance of basis]\label{Invariance} Suppose \(A:=(a^l_k)\) is the transformation between two bases \(D^l\) and \(E^k\) such that \(D^l=\sum_k a^l_k E^k\), then the \(\ksd\) with basis \(D^l\) and the \(\overline{\ksd}\) with basis \(E^k\) satisfies
\begin{equation}\label{KSD-Ineq}
\sqrt{\lambda_{\min}}\cdot \overline{\ksd}\leq \ksd \leq  \sqrt{\lambda_{\max}}\cdot \overline{\ksd},
\end{equation}
where, \(\lambda_{\min}\) and \(\lambda_{\max}\) are the smallest and largest eigenvalues of \(A^\top A\).
\end{theorem}

 \begin{proof}  Let \(\mathcal{T}_p\) be the Stein's operator with basis \(E^k\). Note that we have \(\mathcal{A}^l_p h =\sum_k a^l_k \mathcal{T}^k_p h\) for \(h\in\mathcal{H}_k\). This linear transformation also commutes with Bochner integral, that is,  \(\mathbb{E}_q[\vec{\mathcal{A}}_p f] =A\cdot \mathbb{E}_q[\vec{\mathcal{T}}_p f]\). Recall that \(\ksd=\Vert\mathbb{E}_q[\vec{\mathcal{A}}_p f] \Vert_{\mathcal{H}^d_k}\) and \(\overline{\ksd}=\Vert\mathbb{E}_q[\vec{\mathcal{T}}_p f] \Vert_{\mathcal{H}^d_k}\). The rest of the proof is standard. 
\end{proof}  

In practice, almost all of the commonly used Lie groups are matrix Lie groups, whose Lie algebra (tangent space at identity) is a subspace of $\mathbb{R}^{N\times N}$. Thus, the Lie algebra will inherit the canonical inner product of $\mathbb{R}^{N\times N}$, i.e., $\langle A,B\rangle=\tr(A^\top B)$, and the orthonormal basis w.r.t this inner product is a canonical choice of the basis.

Under mild regularity conditions, the KSD will uniquely characterize \(p\), as stated formally in the following theorem.

\begin{theorem}[Characterization]\label{KSDCharacterization} Suppose \(k\in C^2(G\times G)\) is \(C_0\)-universal. Suppose further \(\sqrt{k_p(x,x)}\) and $D^l \log(p/q)$ are integrable with respect to locally Lipschitz densities \(p\) and \(q\) for all \(l\). Then 
\(p=q\iff\ksd(p,q)=0 \).
\end{theorem}

\begin{proof} The proof is included in \S\ref{Proof-3.4}.
\end{proof}

\textbf{Choice of kernel:} In contrast to \cite{barp2018riemann}, where the \(k\) must reproduce the Sobolev space, a \(C_0\)-universal kernel is much easier to obtain on a Riemannian manifold. Specifically, \cite[Ex. 6.12]{carmeli2010vector} showed that the bivariate function \(\exp(-\frac{\tau}{2}\Vert x-y\Vert^2)\), \(\tau>0\) restricted on any closed set \(X\) in \(\mathbb{R}^m\) is \(C_0\)-universal on \(X\) and we will use it in the next section. It is possible to consider other \(C_0\)-universal kernels such as the Laplace and the Mat\'{e}rn but we will focus on the above mentioned bivaraite function and derive the MKSDE formulas for commonly encountered distributions namely, the exponential family and the Riemannian normal. The derivations will be similar for the other \(C_0\)-universal kernels.
\section{Minimum kernel Stein discrepancy estimator}\label{MKSDE-section}
Given a locally Lipschitz density \(q\) and a family of locally Lipschitz densities \(\{p_\alpha\}\), we define the \emph{minimum kernel Stein discrepancy estimator} (MKSDE) as any of the global minimizers of the KSD, that is,
\begin{equation}
\widehat{\alpha}^{\ksd}:=\argmin\ksd(\alpha):=\argmin\ksd(p_\alpha,q).
\end{equation}
However, the \(\ksd(\alpha)\) is typically not available in practice, as the integral in (\ref{KSDForm}) may be intractable or in some cases we only have samples from \(q\). In such cases, we choose to minimize a discrete \(\ksd\)
from a suite of different \(\ksd\)s based on the situation. In this section, we will introduce these minimization schemes and study their asymptotic properties.

\subsection{Minimization schemes}
The KL-divergence \(D(q||p_\alpha)=\mathbb{E}_q \log(p_\alpha)-\mathbb{E}_q \log(q)\) is commonly used for distribution approximation. If accessibility is limited to the samples from \(q\) instead of its closed form, then, we just minimize the discretized KL-divergence  \(\sum_{i=1}^n \log(p_\alpha(x_i))\), i.e., the log-likelihood function. 

Similarly, the KSD can be discretized as a \(U\)-statistic \(U_n=\frac{1}{n(n-1)}\sum_{i\neq j}k_p(x_i,x_j)\), based on which the kernel Stein goodness of fit test was developed in several prior research works \cite{chwialkowski2016kernel,liu2016kernelized,xu2021interpretable}. We could minimize this \(U\)-statistic to obtain MKSDE from samples, or minimize the \(V\)-statistic
\begin{equation}\label{MKSDE}
\ksd^2_n(\alpha):=\ksd^2_n(p_\alpha,q):= \frac{1}{n^2} \sum_{i,j} k_{p_\alpha}(x_i,x_j),
\end{equation}
which has the analogous asymptotic property but is more stable for optimization since it results in a convex loss function, for the case of $\{p_\alpha\}$ being the von-Mises Fisher family or a mixture of von-Mises Fisher family.

Another situation that commonly occurs in practice is when we do have a closed form of \(p_\alpha\) and \(q\), but the integral in (\ref{KSDForm}) is intractable.
For example, in the rotation tracking problem encountered in robotics \cite{suvorova2021tracking}, one must approximate the posterior distribution of \(Q_k\) with a von Mises-Fisher distribution so that the tracking algorithm (Kalman filter) updates consistently lie in the same space. In a Bayesian fusion problem \cite{lee2018bayesian}, one must go through a similar process to ensure that the result of the fusion stays in the family. In such cases, we can draw samples from \(q\) with any of the various sampling algorithms, e.g., Hamiltonion Monte Carlo or Metropolis-Hastings algorithms, and minimize (\ref{MKSDE}). Alternatively, if these sampling methods are hard to implement, we could use the importance sampling scheme to sample from another distribution \(w\) and minimize the \emph{weighted (discrete) KSD}:
\begin{equation}\label{wMKSDE}
\begin{aligned}
\wksd^2_n(\alpha):= \wksd^2_n(p_\alpha,q)
=
\frac{1}{n^2} \sum_{i,j} k_{p_\alpha}(x_i,x_j) q(x_i) q(x_j)w^{-1}(x_i)  w^{-1}(x_j).
\end{aligned}
\end{equation}

For example, on any compact Lie groups \(G\), e.g., \(\SO(N)\), we could directly sample from the uniform distribution on \(G\), as the volume is finite. Then the KSD in (\ref{wMKSDE}) will degenerate to
\begin{equation}
\wksd^2_n(\alpha)=\frac{1}{n^2} \sum_{i,j} k_{p_\alpha}(x_i,x_j) q(x_i) q(x_j).
\end{equation}
Note that the normalizing constants of \(q\) and \(w\) are not necessary during the minimization.

\subsection{Asymptotic properties of KSD and MKSDE}
In this section, we study the asymptotic properties of KSD and MKSDE obtained from (\ref{MKSDE}) and (\ref{wMKSDE}). Note that the weighted KSD in (\ref{wMKSDE}) will degenerate to KSD in (\ref{MKSDE}) if \(w=q\), thus it suffices to study the MKSDE obtained in (\ref{wMKSDE}). We denote the function
\[ V_\alpha(x,y):=k_{p_\alpha}(x,y) q(x)q(y)w^{-1}(x)w^{-1}(y).
\]
In this section, we assume that all the conditions in theorem \ref{KSDCharacterization} hold. We would like to emphasize that all the results in this section pertain to the weighted KSD in (\ref{wMKSDE}) on Lie groups, and are distinct from the classical results in \cite{barp2019minimum} that involve the unweighted KSD on $\mathbb{R}^d$.

\begin{theorem}[Asymptotic Distribution of KSD]\label{KSDasymptotic} Suppose \(V_\alpha(x,x)\) is \(w\)-integrable. For fixed \(\alpha\), as $n\to\infty$
\begin{enumerate}
    \item \(\wksd_n(\alpha)\to \ksd(\alpha)\) almost surely;
    \item If $\ksd(\alpha)\neq 0$, \[\sqrt{n}[\wksd^2_n(\alpha)-\ksd^2(\alpha)]\xrightarrow{d.} N\big(0,4\Tilde{\sigma}^2\big),\]
    where $\Tilde{\sigma}^2:= \text{var}_{y\sim w}[\mathbb{E}_{x\sim w} V_\alpha(x,y)]$.
    \item If $\ksd(\alpha)=0$, then 
    \[n \wksd^2_n(\alpha)\xrightarrow{d.} \sum_{l=1}^\infty \lambda_l Z^2_l,
    \]
    where \(Z_l\overset{i.i.d}{\sim} N(0,1)\), \(\lambda_l\) are the eigenvalues of the operator \(K: L^2(\omega)\to L^2(\omega), K g := \mathbb{E}_{y\sim w}[V_\alpha(\cdot,y)g(y)]\)
    s.t. \(\sum_{l=1}^\infty \lambda_l=\mathbb{E}_w[V_\alpha(x,x)]\). 
\end{enumerate}
Moreover, we have
\[
\begin{aligned}
  0 \leq \varliminf_{n\to\infty}\inf_\alpha \wksd^2_n(\alpha)\leq  \varlimsup_{n\to\infty}\inf_\alpha \wksd^2_n(\alpha)\leq \inf_\alpha \ksd^2(\alpha)  
\end{aligned}
\]
In particular, if $\inf_\alpha\ksd(\alpha)=0$, then \(\inf_\alpha \wksd_n(\alpha)\to  \inf_\alpha \ksd(\alpha)\) almost surely;
\end{theorem}

\begin{proof}
% Before we embark on the proof, we elaborate on the statement $(3)$ of the theorem. By Cauchy-Schwartz inequality, we have
% $|V_\alpha(x,y)|^2\leq V_\alpha(x,x) V_\alpha(y,y)$, thus 
% \[ 
% \begin{aligned}
% &\iint |V_\alpha(x,y)|^2 \omega(d x) \omega(d y)\\
% &\leq \int V_\alpha(x,x)\omega(d x) \int V_\alpha(y,y)\omega(d y)<+\infty.
% \end{aligned}
% \]
% Therefore, $V_\alpha$ defines a self-adjoint Hilbert-Schmidt integral operator \cite[\S A.5.2 and Thm 4.27]{steinwart2008support} $K:L^2(\omega)\to L^2(\omega)$ by
% \[
% (Kf)(x):=\int V_\alpha(x,y) f(y) \omega(d y).
% \]
% All Hilbert Schmidt operators are compact \cite[\S A.5.2]{steinwart2008support}, thus
% by the spectral theorem \cite[Thm. A.5.13]{steinwart2008support}, there exists an at most countable sequence of real positive numbers (eigenvalues) $\lambda_l$ and an orthogonal basis $\phi_l\in L^2(\omega)$ such that
% \[V_\alpha(x,y)=\sum_{l=1}^\infty \lambda_l \phi_l(x) \phi_l(y).
% \]
% This result is similar to the diagonalization of a symmetric positive definite matrix. 

Note that $V_\alpha(\cdot,\cdot)$ is also positive definite, thus we have  $V_\alpha(x,y)^2\leq V_\alpha(x,x)V_\alpha(y,y)$, thus $\mathbb{E}_q[V_\alpha(x,x)]<+\infty$ implies that $\mathbb{E}_{x,y\sim q}[V_\alpha(x,y)^2]<+\infty$. Therefore, (1), (2) and (3) are straightforward from the classical asymptotic results of $V$-statistics \cite[\S 6.4.1]{serfling2009approximation}. To show the last statement, note that for fixed $\alpha$, $\wksd^2_n(\alpha)=\frac{n-1}{n} U_n+n^{-2}\sum_{i=1}^n V_\alpha(x_i,x_i)$, then $\wksd^2_n(\alpha)\to \ksd^2(\alpha)$ almost surely by SLLN and the strong consistency of $U$-statistic \cite[Thm. 5.4.A]{serfling2009approximation}, with convergence rate of $O_p(n^{-1})$ by the convergence rate of $U$-statistic \cite[Thm 5.4.B]{serfling2009approximation}. Take a sequence of $\alpha_m$ such that $\ksd(\alpha_m)\leq \inf_{\alpha} \ksd(\alpha)+m^{-1} $, then 
\[
\begin{aligned}
\inf_\alpha  \wksd_n(\alpha)\leq \wksd_n(\alpha_m) \to  \ksd(\alpha_m)\leq \inf_\alpha \ksd(\alpha)+m^{-1}.
\end{aligned}
\]
We conclude the result due to the arbitrariness of $m$.
\end{proof}

Up until now, we did not assume any additional structure of the index \(\alpha\). To obtain the asymptotic behavior of the parameter, we re-tag the index of the density family by \(\theta\) from some connected metric parameter space \((\Theta,d)\). Let \(\Theta_0\subset \Theta\) be the set of best approximators \(\theta_0\), i.e., \(\wksd(\theta_0)=\inf_\theta \wksd(\theta)\). Let \(\widehat{\Theta}_n\) be the set of MKSDE, i.e.,  minimizers \(\hat{\theta}_n\) of \(\wksd_n(\theta)\) in (\ref{wMKSDE}), which is a random set. With the additional metric structure on the parameter space, we can establish a stronger asymptotic result on KSD.
\begin{theorem} \label{KSDasymptotic-stronger} Suppose \(V_{(\cdot)}(\cdot,\cdot)\) is jointly continuous and \(\sup_{\theta\in K} V_\theta(x,x)\) is \(w\)-integrable for any compact \(K\subset\Theta\), then \(\wksd^2_n(\theta)\to \ksd^2(\theta) \) compactly almost surely, i.e., for any compact \(K\), as $n\to\infty$,
\[ \wksd^2_n(\theta) \to \ksd^2(\theta) \text{ uniformly on }K, \quad\text{almost surely}. \]
As a corollary, if \(\Theta\) is locally compact, \(\wksd_n\) and \(\ksd\) are all continuous on \(\Theta\).
\end{theorem}

\begin{proof} The proof is included in \S\ref{Proof-4.2}.    
\end{proof}

The condition in theorem \ref{KSDasymptotic-stronger} namely, the joint continuity of $V_{(\cdot)}(\cdot,\cdot)$ , is much easier to test in practice than prior results in \cite{rubin1956uniform,yeo2001uniform} on the uniform strong law for $U$-statistics. 
With theorem \ref{KSDasymptotic-stronger} in hand, we can establish the strong consistency of MKSDE.

\begin{theorem}[Strong consistency]\label{MKSDEconsistency} Suppose the conditions in theorem \ref{KSDasymptotic-stronger} hold and \(\Theta=\Theta_1\times\Theta_2\) such that \(\Theta_1\) is compact, \(\Theta_2\) is convex, and for each fixed \(\theta_1\in\Theta_1\), \(\wksd_n(\theta_1,\cdot)\) is convex on \(\Theta_2\) and \(\ksd(\theta_1,\cdot)\) attains minimum value on a non-empty and compact set \(\Tilde{\Theta}_0(\theta_1)\subset \text{int}(\Theta_2)\). Then \(\Theta_0\), \(\widehat{\Theta}_n\) are non-empty for large \(n\) and \(\sup_{\theta\in \widehat{\Theta}_n} d(\theta,\Theta_0)\to 0\) almost surely as $n\to \infty$.
\end{theorem}

\begin{proof} The proof is included in \S\ref{Proof-4.3}.    
\end{proof}

It is worth noting that if the family $p_\theta$ is identifiable and $q$ is a member of the family \(p_\theta\), then the set \(\Theta_0\) of global minimizers is a singleton \(\{\theta_0\}\). In such a situation, the MKSDE always converges to the unique \(\theta_0\). 

In \cite[Thm. 3.3]{barp2019minimum} Barp  et al., showed the consistency of MKSDE, assuming that the parameter space \(\Theta\) is either compact or satisfies the conditions of convexity. However, this is not applicable to the Riemannian normal distribution, as it has a compact \(\mu\in\SO(N)\) and a convex \(\tau>0\) simultaneously. In contrast, theorem \ref{KSDasymptotic-stronger}, as a stronger version, only requires checking for the existence of global minimizers for one component, instead of for both components, and thus applicable to this situation.

In addition, note that the MKSDE of vMF is a quadratic form of $F$. Therefore, we have:

\begin{corollary} The MKSDE for von Mises-Fisher distribution computed via (\ref{KSD-SO(N)-vMF}) and (\ref{wMKSDE}), and the MKSDE for Riemannian normal distribution computed via (\ref{KSD-SO(N)-RN}) and (\ref{wMKSDE}) are strongly consistent.
\end{corollary} 
\begin{proof} The proof is straightforward.    
\end{proof}

To establish the asymptotic normality of MKSDE, we assume that \(\Theta\) is a connected Riemannian manifold with a Levi-Civita connection \(\nabla\) and the Riemannian logarithm map, \(\Log\). We assume that \(\Theta_0\) and \(\widehat{\Theta}_n\) are non-empty for \(n\) large enough, and $\hat{\theta}_n$ is a sequence of MKSDE that converges to one of the global minimizer $\theta_0$ of $\ksd(\theta)$. Additionally, we assume the following conditions:
\begin{enumerate}
    \item[\textbf{(A1)}] $V_\theta(x,y)$ is jointly continuous, and twice continuously differentiable in $\theta$.
    \item[\textbf{(A2)}] there exists a neighborhood $U$ of $\theta_0$ such that $\sup_{\theta\in U}\Vert \nabla V_{\theta_0}(x,y)\Vert$ is $w\times w$-integrable.
    \item[\textbf{(A3)}]$\Vert \nabla V_{\theta_0}(x,y)\Vert^2$ and $\Vert\mathcal{I}_{\theta_0}(x,y)\Vert$ are $w\times w$-integrable, $\Vert\mathcal{I}_{\theta_0}(x,x)\Vert$ is $w$-integrable.
    \item[\textbf{(A4)}] $\mathcal{I}_{\theta_0}(x,y)$ is equi-continuous at $\theta_0$. 
    \item[\textbf{(A5)}] $\Gamma:=\frac{1}{2}\mathbb{E}_{x,y\sim w} [\mathcal{I}_{\theta_0}(x,y)]$ is invertible.
\end{enumerate}
Here $\nabla V_\theta(x,y)$ represents the gradient of $V_\theta(x,y)$ w.r.t $\theta$, and $\mathcal{I}_\theta(x,y)$ represents the Hessian of $V_\theta(x,y)$ w.r.t $\theta$. In addition, let $\Sigma$ be the covariance matrix of the random vector $\mathbb{E}_{Y\sim w}[V_{\theta_0}(x,Y)]$.

\begin{theorem}[CLT for MKSDE]\label{MKSDECLT} Under assumption $\textbf{A1}$, $\textbf{A2}$, $\textbf{A3}$, $\textbf{A4}$ and $\textbf{A5}$, we have
 \(\sqrt{n}\Log_{\theta_0} (\hat{\theta}_n) \xrightarrow{d.} \mathcal{N}(0,\Gamma^{-1}\Sigma\Gamma^{-1})\). 
\end{theorem}
\begin{proof}  In this work, the parameter space is not assumed to be a vector space and thus is not necessarily flat. However, as the MKSDE $\hat{\theta}_n$ converges to the ground truth $\theta_0$, the sequence will finally enter a neighborhood of $\theta_0$. Since a manifold is locally Euclidean, the parameter space $\Theta$ can be considered as a flat space without loss of generality. It is convenient to take the normal coordinates given by the exponential map Exp$_{\theta_0}$ at $\theta_0$, as the Jacobian of the Exp$_{\theta_0}$ is $I$ at $0$ and the geodesic connecting $\hat{\theta}_n$ with $\theta_0$ coincides with the line segments connecting $\Log_{\theta_0}(\theta)$ with $0$ on the tangent space at $\theta_0$. The rest of the proof is the same as the one given in \cite[Thm. 12]{oates2022minimum}.
\end{proof}

\subsection{Applications}\label{Applications}

\subsubsection{MKSDE goodness of fit test for a distribution family}\label{MKSDE-GoF}

The goodness of fit test is a hypothesis test that tests whether a group of samples can be well-modeled by a given distribution \(p\), that is, if we denote by \(q\) the unknown underlying distribution of samples, we aim to test \(H_0:p=q\) versus \(H_1:p\neq q\). Several works \cite{chwialkowski2016kernel,liu2016kernelized} utilized the KSD to develop normalization-free goodness of fit tests on \(\mathbb{R}^d\). However, their methods only apply to a specific candidate \(p\), and requires testing individually for each member of a parameterized family \(p_\theta\). In many applications, the candidate distribution for the given samples is usually not a specific distribution but a parameterized family. In this section, \emph{we develop a one-shot method to test whether a group of samples \(x_i\) matches a family \(p_\theta\), using the MKSDE obtained by minimizing (\ref{MKSDE})}.

The \emph{MKSDE goodness of fit} performs the test \(H_0:\exists\theta_0, p_{\theta_0}=q\) versus \(H_1:\forall\theta, p_\theta\neq q\). Under the null hypothesis $H_0$,  \(n\wksd^2_n(\theta_0)\sim\sum_{j=1}^\infty\lambda_i Z_j^2\) asymptotically by (3) in theorem \ref{KSDasymptotic}. Let \(\gamma_{1-\beta}\) be the \((1-\beta)\)-quantile of \(\sum_{j=1}^\infty \lambda_j Z^2_j\) with significance level \(\beta\).  We reject \(H_0\) if \(n\wksd_n(\hat{\theta}_n)\geq \gamma_{1-\beta}\), as it implies \(n\wksd^2_n(\theta_0)\geq n\wksd^2_n(\hat{\theta}_n) \geq \gamma_{1-\beta} \), since \(\hat{\theta}_n\) minimizes \(\wksd^2_n\). 

The evaluation of \(\gamma_{1-\beta}\), relies on following result:

\begin{proposition}\label{Classical-GoF} 
Let \(\hat{\lambda}'_l\) be the eigenvalues of the Gram matrix \(n^{-1}(V_{\theta_0}(x_i,x_j))_{i j}\) (set $\hat{\lambda}'_l:=0$ for $l>n$). Suppose $V_{\theta_0}(x,x)$ is $p_{\theta_0}$-integrable. Then $\sum_{l=1}^\infty(\hat{\lambda}'_l-\lambda_l) Z^2_l \to 0$ in probability as $n\to\infty$.
\end{proposition}

\begin{proof} See \cite[Theorem 1]{gretton2009fast}.    
\end{proof}

Therefore, the $\sum_{l=1}^n \hat{\lambda}_l Z^2_l$ can serve as an empirical estimate of the asymptotic distribution $\sum_{l=1}^\infty \lambda_l Z^2_j$. 

Although the ground truth $\theta_0$ is unknown in our setting, the minimizers $\hat{\theta}_n$ of $\wksd^2_n$ converge to $\theta_0$ almost surely by theorem \ref{MKSDEconsistency} under specific conditions. Moreover, if the function $V_\theta$ is continuous in $\theta$, then we can use the eigenvalues \(\hat{\lambda}_l\) of \(n^{-1}(V_{\hat{\theta}_n}(x_i,x_j))_{i j}\) as an alternative. To sum up, we have

\begin{theorem}\label{GoF} Suppose the conditions in theorem \ref{MKSDEconsistency} hold and $\{p_\theta\}$ is identifiable, then $\sum_{l=1}^\infty (\hat{\lambda}_l-\lambda_l) Z^2_l\to 0$ in probability as $n\to\infty$.
\end{theorem}

\begin{proof} The proof is included in \S\ref{Proof-GoF}.    
\end{proof}

The MKSDE goodness of fit test algorithm is summarized in the algorithm block \ref{MKSDE-Algo}. To implement the test, we assume the conditions in theorem $\ref{MKSDEconsistency}$ hold and $\{p_\theta\}$ is identifiable.

\begin{algorithm}[H]
\caption{MKSDE goodness of fit test.}\label{MKSDE-Algo}
\begin{algorithmic}
\STATE 
\STATE \textbf{Input}: population \(x_1,\dots,x_n\sim q\); sample size \(n\); number of generations \(m\); significance level \(\beta\).
\STATE \textbf{Test:} \(H_0:\exists\theta_0, p_{\theta_0}=q\) versus \(H_1:\forall\theta,\    p_\theta \neq q\).
\STATE \textbf{Procedure:}
\STATE \hspace{0.5cm} 1. Find the minimizer \(\hat{\theta}_n\) of \(\wksd^2_n(\theta)\) respectively either numerically using gradient descent or analytically as shown in the example \ref{vMF-SO(N)}. 
\STATE \hspace{0.5cm} 2. Obtain the eigenvalues \(\hat{\lambda}_1,\dots, \hat{\lambda}_n\) of Gram matrix \(n^{-1}(V_{\hat{\theta}_n}(x_i,x_j))_{i j}\).
\STATE \hspace{0.5cm} 3. Sample \(Z^{i}_j\sim N(0,1)\), \(1\leq j\leq n\), \(1\leq i\leq m\) independently.
\STATE \hspace{0.5cm} 4. Compute \(W^i=\sum_{j=1}^n \hat{\lambda}_j (Z^i_j)^2 \).
\STATE \hspace{0.5cm} 5. Determine estimation \(\hat{\gamma}_{1-\beta}\) of \((1-\beta)\)-quantile using \(W^1,\dots,W^m\).
\STATE \textbf{Output:} Reject \(H_0\) if \(n\cdot \wksd^2_n(\hat{\theta}_n)>\hat{\gamma}_{1-\beta}\). 
\end{algorithmic}
\label{alg1}
\end{algorithm}

\subsubsection{KSD-EM algorithm}\label{KSD-EM}
The EM algorithm is implemented to estimate the parameter \(\theta\) of a probabilistic model \(x\sim p(x|z,\theta)p(z|\theta)\) with an unobserved latent variable \(z\), e.g., as in the probabilistic principal geodesic analysis (PPGA) technique \cite{zhang2014bayesian,zhang2013probabilistic,zhang2019mixture} on Riemannian manifolds. In the classical EM algorithm, with samples \(\boldsymbol{x}\), one generates an iterate \(\theta^t\) of \(\theta\) by maximizing the log-likelihood \(Q(\theta|\boldsymbol{x},\theta^{t})\) with \(z\) marginalized, i.e., \(Q(\theta|\theta^{t}):= \mathbb{E}_{z\sim p(z|\boldsymbol{x},\theta^t) } [\log p(\theta|z,\boldsymbol{x},\theta^t)]\). The expectation is estimated by \(n^{-1}\sum_i\log p(z_i,\boldsymbol{x}|\theta^t)\), with samples \(z_i\) generated from \(p(z|\boldsymbol{x},\theta^t)\).

In applications, as the normalizing constant is most often intractable, one will encounter difficulty in either generating the samples \(z_i\) or maximizing the log-likelihood. However, one can now replace the log-likelihood loss with the weighted KSD from (\ref{wMKSDE}) and cope with the issue of the intractable normalizing constants in both \(p(z|\boldsymbol{x},\theta^t)\) and \(p(\theta|z,\boldsymbol{x},\theta^t)\). That is, sample $z_1,\dots ,z_n$ from the distribution $\omega$, and minimize
\begin{equation}\label{QwKSD}
\begin{aligned}
Q_{\wksd}(\theta|\theta^t):=\frac{1}{n^2}\sum_{i,j=1}^n k_{p(z|\theta,\boldsymbol{x},\theta^t)}(z_i,z_j)\cdot p(z_i|\boldsymbol{x},\theta^t) p(z_j|\boldsymbol{x},\theta^t) \omega(z_i)^{-1}\omega(z_j)^{-1}
\end{aligned}
\end{equation}
in one shot.

The KSD-EM algorithm is summarized
in the algorithm block 2.

\begin{algorithm}[H]
\caption{KSD-Expectation Maximization}\label{KSDEM-Algo}
\begin{algorithmic}
\STATE 
\STATE \textbf{Input}: population \(\boldsymbol{x}\sim q\);  initial value $\theta^0$; iteration count $t=0$; error tolerance $\epsilon$.
\WHILE{$|Q_{\wksd}(\theta^t|\theta^{t-1})-Q_{\wksd}(\theta^{t-1}|\theta^{t-2})|>\epsilon$}
\STATE \hspace{0.5cm} 1. Sample $z_1,\dots,z_n\sim \omega$.
\STATE \hspace{0.5cm} 2. Compute $\theta^{t+1}:=\argmin_\theta Q_{\wksd}(\theta|\theta^t)$ numerically or analytically.
\STATE \hspace{0.5cm} 3. $t\leftarrow t+1$.
\ENDWHILE

\STATE \textbf{Output:} $\theta^t$.
\end{algorithmic}
\label{alg2}
\end{algorithm}

\section{KSD and MKSDE on \texorpdfstring{\(\SO(N)\)}{SO(N)}} \label{ExpSO(N)}

As \(\SO(N)\) is one of the most widely encountered Lie group in applications, we demonstrate the mechanics of deriving  the function \(k_p(\cdot,\cdot)\) in (\ref{KSDForm}) and closed form of MKSDE obtained by minimizing (\ref{wMKSDE}) for commonly used distribution families on \(\SO(N)\) namely, the exponential family and the Riemannian normal distribution. 

\paragraph{Notations} For notational convenience, we set up following notations in this section:
\begin{center}
\begin{tabular}{ |c|l|   } 
  \hline
  $\mathscr{A}(A)$ & the skew-symmetrization \(\mathscr{A}(A)=(A-A^\top)/2\) of a squared matrix $A$ \\ 
  \hline
  $\vectz(A)$ & vectorization of a matrix $A$ by stacking the columns \\ 
  \hline
  $\otimes$ & the Kronecker product, see \cite{van2000ubiquitous} for example
  \\
  \hline
  $\Log$ & the matrix logarithm \\ 
  \hline
  \multirow{2}{*}{$S_{N,N}$} & the perfect shuffle matrix \cite{van2000ubiquitous},  defined as \((S_{N,N})_{a b}=1\) whenever \(a=Nk+l-N\) \\
  & and \(b=Nl+k-N\) for some \(1\leq k,l\leq N\), and equals \(0\) otherwise. 
  \\
  \hline
  \multirow{2}{*}{$\nabla_X$} & the Euclidean gradient $\nabla_X f\in\mathbb{R}^{N\times N}$ of a differentiable function on $\SO(N)$ is defined as\\
  & a $N\times N$ matrix such that $Df(X)=\tr[D^\top\nabla_X f]$ for all $D\in T_X\SO(N)$.
  \\
  \hline
\end{tabular}
\end{center}

\paragraph{Commonly-used families} We now introduce the commonly-used distribution families in Lie Groups:

\begin{itemize}
    \item {\it Exponential family}: 
    \[
p(X|\theta)\propto \exp(\theta^\top \zeta(X)+\eta(X)), \quad \theta\in\mathbb{R}^m, m\in\mathbb{N},
\]
where $\zeta :\SO(N)\to\mathbb{R}^m$ and $\eta:\SO(N)\to\mathbb{R}$ are continuously differentiable. Note that the widely used von Mises-Fisher family given by,
\[
p(X|F)\propto \exp(\tr(F^\top X)), \quad F\in\mathbb{R}^{N\times N},
\]
is a member of the exponential family, as $\tr(F^\top X)=\vectz(F)^\top \vectz(X)$. 
 \item {\it Riemannian normal family}:
    \[
p(X|\Bar{X},\sigma)\propto \exp(-\frac{1}{2 \sigma^2}\tr(\Log(\Bar{X}^\top X)^\top \Log(\Bar{X}^\top X))),\quad \Bar{X}\in \SO(N), \sigma>0.
\]
It should be noted that the Riemannian normal is \emph{not} a member of the exponential family as it uses the intrinsic distance function on $\SO(N)$.
\end{itemize}

\paragraph{Choice of kernel on $\SO(N)$} 
Earlier, we chose as our $C_0$-universal kernel, the bivariate function $\exp(-\tau/2 ||x-y||^2), \tau > 0$ restricted to a closed subset $X \subset \mathbb{R}^m$. We now choose $X$ to be $\SO(N)$,
since $\SO(N)$ is closed in \(\mathbb{R}^{N\times N}\), the Gaussian kernel \(e^{-\frac{\tau}{2}\tr((X-Y)^\top (X-Y))}\) restricted to \(\SO(N)\) is \(C_0\)-universal, which becomes
\[
k(X,Y)=\exp(\tau\tr(X^\top Y)), \quad X,Y\in\SO(N),
\] 
where \(\tau>0\) is arbitrary. As mentioned earlier, the key to the choice of the kernel is the \(C_0\)-universality condition and  one can choose other kernels such as the Laplace or Mat\'{e}rn, both of which can be shown to satisfy the universality condition based on the results in \cite{Sriperumbudur2008COLT,Sriperumbudur2011JMLR}. The
recipe for the derivation of closed form expressions for the restriction of these kernels to $\SO(N)$ and the MKSDE is rather complex and tedious but in principle similar to that presented here. Hence we will only focus on the above chosen restriction of the Gaussian kernel to $\SO(N)$.

\paragraph{Vector field basis on $\SO(N)$} Let \(E_{i j}\), \(1\leq i< j\leq N\) be the matrix with all zeros except \(\sqrt{2}/2\) at the \((i,j)^{th}\) entry and \(-\sqrt{2}/2\) at \((j,i)^{th}\) entry of the matrix. Then \(X E_{i j}\), \(X\in \SO(N)\), is a standard orthonormal left-invariant basis on \(\SO(N)\).

\subsection{\texorpdfstring{$k_p$}{kp} functions on \texorpdfstring{\(\SO(N)\)}{SO(N)}}

Let $\mathcal{A}^{i j}_p$ represent the operator component corresponding to $XE_{i j}$. Clearly,
\begin{equation}\label{Kpcomponents}
 \begin{aligned}
\Tilde{\mathcal{A}}^{i j}_p \mathcal{A}^{i j}_p k(X,Y)&= \tr[(\nabla_X \log p+\tau Y)^\top X E_{i j}] \cdot \tr[(\nabla_Y \log p+ 
\tau X)^\top Y E_{i j}]\cdot e^{\tau\tr(X^\top Y)} \\ 
&+ \tau\tr(E^\top_{i j}X^\top Y E_{i j}) \cdot e^{\tau\tr(X^\top Y)}. 
 \end{aligned}
\end{equation}
Recall that $E_{i j}$ is an orthonormal basis of the space Skew$(N)$ of all skew-symmetric matrices with the inner product $\langle \cdot,\cdot \rangle:=\tr( \cdot^\top \cdot )$. Note that $\tr(A^\top E_{i j} )$ is the inner product of $A$ and $E_{i j}$, which actually equals $\tr(\mathscr{A}(A)^\top E_{i j})$, as $\mathscr{A}(A)$ is the orthogonal projection of $A$ onto the space Skew$(N)$. Furthermore, for any two matrices $A$ and $B$, 
\[
\begin{aligned}
    \sum_{i<j}\tr(A^\top E_{i j})\tr(B^\top E_{i j})&=
    \sum_{i<j}\tr(\mathscr{A}(A)^\top E_{i j})\tr(\mathscr{A}(B)^\top E_{i j})= 
    \sum_{i<j} \langle \mathscr{A}(A),E_{i j} \rangle\cdot \langle \mathscr{A}(B),E_{i j} \rangle\\
    &= \langle\mathscr{A}(A), \mathscr{A}(B)\rangle = \tr[\mathscr{A}(A)^\top \mathscr{A}(B)].
\end{aligned}
\]
Therefore, the first term on the RHS in (\ref{Kpcomponents}) %sum up to
becomes
\[ 
\begin{aligned}
\tr[\mathscr{A}(X^\top(\nabla_X \log p +\tau Y))^\top \mathscr{A}(Y^\top(\nabla_Y \log p+\tau X)) ] \cdot e^{\tau \tr (X^\top Y)}.
\end{aligned}
 \]
Moreover, note that $E_{i j} E^\top_{i j}$ equals the matrix which has all $0$s except two $(1/2)$s at $(i,i)^{th}$ and $(j,j)^{th}$ entry respectively. The sum of $E_{i j} E^\top_{i j}$, $1\leq i<j\leq N$ actually equals $\frac{N-1}{2} I$. Recall that one can commute matrices inside $\tr(\cdot)$, thus the second term on the RHS of (\ref{Kpcomponents}) %sum up to
becomes,
\[ c(X,Y):=\frac{\tau}{2}(N-1)\tr(X^\top Y) e^{\tau \tr (X^\top Y)}. \]
Combining the two terms we obtain 
\begin{equation}\label{Kp-SO(N)-Eq}
\begin{aligned}
k_p(X,Y)
=\tr[\mathscr{A}(X^\top(\nabla_X \log p +\tau Y))^\top \mathscr{A}(Y^\top(\nabla_Y \log p+\tau X)) ] \cdot e^{\tau \tr (X^\top Y)}+ \frac{\tau}{2}(N-1)\tr(X^\top Y) e^{\tau \tr (X^\top Y)}
\end{aligned}
\end{equation}
Note that the last term, denoted as \(c:=c(X,Y)\) in what follows is independent of the distribution \(p\), and hence the parameters of $p$. We are now ready to substitute the expression of the gradient of $\log p$  for the exponential family and the von Mises-Fisher in particular, as well as the Riemannian normal into the above expression for the chosen kernel on $\SO(N)$ leading to closed form expressions of the kernel for each case.

\paragraph{Exponential family} Substituting the Euclidean gradient in to (\ref{Kp-SO(N)-Eq}) we have
\begin{equation}
k^{\text{Exp}}_\theta(X,Y)=c(X,Y)+\tr[\mathscr{A}(X^\top(\nabla_X(\theta^\top \zeta(X)+ \eta(X))+\tau Y))^\top \mathscr{A}(Y^\top(\nabla_Y(\theta^\top \zeta(Y)+ \eta(Y))+\tau X)) ] \cdot e^{\tau \tr (X^\top Y)}.
\end{equation}

\paragraph{von Mises-Fisher} The Euclidean gradient of
$\nabla_X\log p=\tr(F^\top X)$ is exactly $F$. 5Substituting this in to (\ref{Kp-SO(N)-Eq}) we have,
\begin{equation}\label{KSD-SO(N)-vMF}
\begin{split}
k^{\vmf}_F(X,Y) = c(X,Y)+\tr[\mathscr{A}(X^\top(F+\tau Y))^\top \mathscr{A}(Y^\top(F+\tau X)) ] \cdot e^{\tau \tr (X^\top Y)}.
\end{split}
\end{equation}

\paragraph{Riemannian normal} Let \(\varsigma=\sigma^{-2}\). The Euclidean gradient of $\log p = -\frac{\varsigma}{2} \tr[\Log(\Bar{X}^\top X)\Log(\Bar{X}^\top X)] $ at $X$ is $\nabla_X\log p = \varsigma X\Log(X^\top\Bar{X})$. Substituting this in to (\ref{Kp-SO(N)-Eq}) leads to the following form:
\begin{equation}\label{KSD-SO(N)-RN}
\begin{split}
k^{\text{RN}}_{\Bar{X},\varsigma}= c+\tr[(\varsigma\Log(X^\top\Bar{X})+\tau \mathscr{A}(X^\top Y))^\top (\varsigma\Log(Y^\top\Bar{X})+\tau \mathscr{A}(Y^\top X)) ] \cdot e^{\tau \tr (X^\top Y)}.
\end{split}
\end{equation}

\subsection{Closed Form of MKSDE for the exponential family on \texorpdfstring{$\SO(N)$}{SO(N)}}\label{MKSDE-Exp-SO(N)}

A closed form solution of MKSDE will evidently reduce the computational cost. Apparently, there is no general solution for any distribution family, since the parametrization $\alpha\mapsto p_\alpha$ of a family, as a map from the parameter space to the space of distributions, can be in any one of many different forms. 

However, it is noteworthy that if $\log p_\theta$ is linearly parametrized by the parameter $\theta$, i.e., $\log p_\theta= \zeta(X)^\top\theta+\eta(X)+C$, for some $\zeta$ and $\eta$ (which gives us exactly the exponential family), then the $k_p$ function in (\ref{Kp-SO(N)-Eq}) on $\SO(N)$ is a quadratic form of $\theta$, so that the MKSDE obtained by minimizing the weighted KSD in (\ref{wMKSDE}) will have a closed form. 

In this section, we derive the closed form of MKSDE for the exponential family on $\SO(N)$.
For notational convenience, we denote $\Pi_X=\nabla_X\log p(X)$. First, we split $k_p$ in (\ref{Kp-SO(N)-Eq}) into several terms of different orders of $\Pi_X$ and $\Pi_Y$:
\[
\begin{aligned}
k^{\text{Exp}}_\theta(X,Y)&=\tr[\mathscr{A}(X^\top \Pi_X)^\top\mathscr{A}(Y^\top\Pi_Y)]+\tau\tr[\mathscr{A}(Y^\top X)^\top X^\top \Pi_X ] +\tau\tr[\mathscr{A}(X^\top Y)^\top Y^\top \Pi_Y ] + C'\\
&= \underbrace{\frac{1}{2} \tr[\Pi_X^\top X Y^\top \Pi_Y ]}_{(I)}-\underbrace{\frac{1}{2}\tr[ X^\top \Pi_X Y^\top \Pi_Y ]}_{(II)}+\tau\underbrace{\tr[\mathscr{A}(Y^\top X)^\top X^\top \Pi_X ]}_{(III)} +\tau\underbrace{\tr[\mathscr{A}(X^\top Y)^\top Y^\top \Pi_Y ]}_{(IV)} + C'
\end{aligned}
\]
The constant $C'$ is independent of the $\Pi_X$ and $\Pi_Y$, thus also independent of the parameters of $p$. The second-order terms will equal
\[
\begin{aligned}
(I) &=\frac{1}{2}[(I\otimes X^\top) \vectz(\Pi_X)]^\top [(I\otimes Y^\top)\vectz(\Pi_Y)]=\frac{1}{2}\vectz(\Pi_X)^\top(I\otimes X Y^\top) \vectz(\Pi_Y),\\
(II)&=
\frac{1}{2}\vectz(\Pi_X)^\top (Y^\top\otimes X) \vectz(\Pi_Y^\top)=\frac{1}{2}\vectz(\Pi_X)^\top(Y^\top\otimes X) S_{N,N} \vectz(\Pi_Y).
\end{aligned}
\]
The first-order terms will equal 
\[
\begin{aligned}
(III) &=\vectz( X\mathscr{A}(Y^\top X))^\top\vectz(\Pi_X),\\
(IV)&= \vectz(Y\mathscr{A}(X^\top Y))^\top\vectz(\Pi_Y).
\end{aligned}
\]
Combining all the terms, we have
\begin{equation}\label{kp-vec-Pi}
\begin{aligned}
k^{\text{Exp}}_\theta(X,Y)&=C'+\frac{1}{2} \vectz(\Pi_X)^\top (I\otimes XY^\top-Y^\top\otimes X\cdot S_{N,N})\vectz(\Pi_Y)\\
&+\tau \vectz( X\mathscr{A}(Y^\top X))^\top\vectz(\Pi_X)+\tau\vectz(Y\mathscr{A}(X^\top Y))^\top\vectz(\Pi_Y).
\end{aligned}
\end{equation}

Next, we represent $\Pi_X$ and $\Pi_Y$ by $\theta$. As $\theta^\top \zeta(X)$ is a scalar functions on $\SO(N)$, its Euclidean gradient $\nabla_X (\theta^\top \zeta(X)) $ is a $N\times N$ matrix. However, if we extract $\theta$ from the Euclidean gradient, i.e., $\nabla_X (\theta^\top \zeta(X))=\theta^\top\nabla_X\zeta(X)$, then the Euclidean gradient $\nabla_X\zeta(X)$ will be a \emph{3D matrix} of dimension $m\times N\times N$, since $\zeta(X)$ is a vector-valued function of $X$. To appropriately tackle this issue, we vectorize the  $N\times N$ dimensional component of
$\nabla_X\zeta(X)$, so that it is represented by a $m\times N^2$-dimensional 2D matrix, as elaborated upon next.

We denote $\zeta(X):=(\zeta_1(X),\dots,\zeta_m(X))^\top$, then each of $\nabla_X\zeta_i(X)$ is a $N\times N$ matrix. We vectorize each $\nabla_X \zeta_i(X)$ and stack them by rows to get 
\[
\mathcal{Z}(X):= 
\begin{bmatrix}
\vectz(\nabla_X\zeta_1(X))^\top\\
\vdots
\\
\vectz(\nabla_X\zeta_m(X))^\top
\end{bmatrix} \in \mathbb{R}^{m\times N^2}.
\]
Note that $Z(X)^\top\theta =\vectz(\nabla_X[\theta^\top \zeta(X)])$. Therefore, we have $\vectz(\Pi_X)= \mathcal{Z}(X)^\top \theta+\vectz(\nabla_X\eta(X))$. Substituting this into (\ref{kp-vec-Pi}) we get
\begin{equation}\label{kp-vec-theta}
\begin{aligned}
k^{\text{Exp}}_\theta(X,Y)&=\frac{1}{2}\theta^\top \mathcal{Z}(X) (I\otimes XY^\top-Y^\top\otimes X\cdot S_{N,N}) \mathcal{Z}(Y)^\top \theta\\
&+ \frac{1}{2}\vectz(\nabla_X\eta(X))^\top(I\otimes XY^\top-Y^\top\otimes X\cdot S_{N,N}) \mathcal{Z}(Y)^\top \theta
\\
&+ \frac{1}{2}\vectz(\nabla_Y\eta(Y))^\top(I\otimes YX^\top- X^\top \otimes Y\cdot S_{N,N} ) \mathcal{Z}(X)^\top \theta
\\
&+\tau \vectz( X\mathscr{A}(Y^\top X))^\top\mathcal{Z}(X)^\top \theta +\tau\vectz(Y\mathscr{A}(X^\top Y))^\top\mathcal{Z}(Y)^\top \theta.
\end{aligned}
\end{equation}
Although above formula is rather lengthy and monstrous, it is noteworthy that the first-order terms can be combined together by the summation in the weighted KSD in (\ref{wMKSDE}), as they are symmetric with respect to $X$ and $Y$. We summarize the results for the exponential family in the following theorem:
\begin{theorem}\label{Exp-SO(N)} Suppose $X_i\in\SO(N)$ are samples from $w$ and $p_\theta$ is the exponential family on $\SO(N)$. Given kernel $k(X,Y)=\exp(\tau\tr(X^\top Y))$ on $\SO(N)$, the global minimizer of the weighted KSD in (\ref{wMKSDE}) has a closed form given below:
\begin{equation}\label{Exp-SO(N)-Eq}
\begin{aligned}
\text{Let} \ b&=\frac{1}{2n^2} \sum_{i,j} \vectz(\nabla_{X_i}\eta(X_i))^\top(I\otimes X_i X_j^\top-X_j^\top\otimes X_i\cdot S_{N,N}) \mathcal{Z}(X_j)^\top\cdot e^{\tau \tr(X_i^\top X_j)} \frac{q(X_i)}{w(X_i)}\cdot \frac{q(X_j)}{w(X_j)} \\
&+ \frac{\tau}{n^2} \sum_{i,j} \vectz(X_i\mathscr{A}(X_j^\top X_i))^\top \mathcal{Z}(X_i)^\top \cdot e^{\tau \tr(X_i^\top X_j)} \frac{q(X_i)}{w(X_i)}\cdot \frac{q(X_j)}{w(X_j)}
\\
\text{and} \ A&=  \frac{1}{2 n^2}\sum_{i,j} \mathcal{Z}(X_i) (I\otimes X_i X_j^\top-X_j^\top\otimes X_i\cdot S_{N,N}) \mathcal{Z}(X_j)^\top \cdot e^{\tau \tr(X_i^\top X_j)} \frac{q(X_i)}{w(X_i)}\cdot \frac{q(X_j)}{w(X_j)}.
\end{aligned}  
\end{equation}
Then, $\hat{\theta}_n= -A^{-1} b$. For the unweighted case, we just ignore the weighted ratio \(\frac{q(X_i)q(X_j)}{w(X_i)w(X_j)}\).
\end{theorem}
\begin{proof}
Since (\ref{kp-vec-theta}) is a quadratic function of $\theta$, the weighted KSD in (\ref{wMKSDE}), 
which is a weighted sum of $k_p$ function in (\ref{kp-vec-theta}),
will remain a quadratic function of $\theta$. Therefore, the global minimizer is $\hat{\theta}_n=-A^{-1} b$. 
\end{proof}

\begin{example}[MKSDE of vMF]\label{vMF-SO(N)} We now consider the MKSDE of the von Mises-Fisher family on $\SO(N)$. Note that $\log p = \tr(F^\top X)+C=\vectz(F)^\top\vectz(X)$, so that von Mises-Fisher family is a member of the exponential family if we set $\theta:=\vectz(F)$. Since $\zeta(X)=\vectz(X)$ and $\eta(X)=0$ in this case, we have $\mathcal{Z}(X)=I_{N^2\times N^2}$ and $\nabla_X\eta(X)=0$. Substituting this into (\ref{Exp-SO(N)}) yields the closed form of MKSDE given below,
\begin{equation}\label{vMF-SO(N)-Eq}
\begin{aligned}
\text{Let} \ b&= \frac{\tau}{n^2}\sum_{i,j} \vectz[X_i\mathscr{A}(X_i^\top X_j)] e^{\tau \tr(X_i^\top X_j)}\cdot \frac{q(X_i)}{w(X_i)}\cdot \frac{q(X_j)}{w(X_j)}, \\
\text{and} \ A&=  \frac{1}{2 n^2}\sum_{i,j} [I\otimes X_i X_j^\top - (X_i^\top\otimes X_j) S_{N,N}]\cdot e^{\tau\tr( X_i^\top X_j) } \cdot \frac{q(X_i)}{w(X_i)}\cdot \frac{q(X_j)}{w(X_j)}.
\end{aligned}  
\end{equation}
Then, \(\widehat{F}_{\wksd}= \vectz^{-1}(A^{-1} b)\). For the unweighted case, we just ignore the weighted ratio \(\frac{q(X_i)q(X_j)}{w(X_i)w(X_j)}\).
\end{example}
\section{Experiments}\label{experiments}
In this section we present two experiments to validate the advantage of our KSD over MLE. In the first experiment, we compare the MKSDE of \(F\), a parameter of the vMF presented in example \ref{vMF-SO(N)}, to its MLE, illustrating that the issue of normalizing constant impairs the accuracy of MLE but has no affect on the MKSDE. In this experiment, based on estimated parameter $F$, we also estimate the orientation of a 3D object from a publicly available database, ModelNet10 \cite{wu20153d}. In the second experiment, we compare the Cayley distribution \cite{leon2006statistical} to the vMF via a goodness of fit test introduced in \S\ref{MKSDE-GoF} to illustrate the power of our test. In these experiments, we set $\tau=1$. Code for all the experiments in this
paper is provided on GitHub at \url{https://github.com/cvgmi/KSD-on-Lie-Groups}.

\subsection{MKSDE vs. MLE}\label{MKSDEvsMLE}

\paragraph{vMF Parameter Estimation} The exact numerical solution of MLE for vMF requires the computation of the inverse of the derivative of the normalizing constant, a hypergeometric function of the parameter \(F\). The commonly used MLE technique \cite[\S 13.2.3]{mardia2000directional} uses two direct approximate solutions, one is for the case when \(F\) is small while the other is for large \(F\). Although the solution for small \(F\) is reasonably accurate, the solution method for large \(F\) is cumbersome and hard to implement, and both solutions poorly approximate for in-between values of \(F\). Therefore, one must bear with either the inaccuracy of a direct approximate solution or the computational cost of achieving an convergent numerical solution.

Figure \ref{MKSDE&MLE} depicts the Frobenius distance of \(\hat{F}_{\ksd}\) and \(\hat{F}_{\mle}\) to the ground truth \(F_0\), with varying values of $F_0$ and varying sample size \(n\). For medium valued \(F_0\), e.g., \(F=5I\), we used the approximation for small \(F_0\). The last "random" \(F_0\) is drawn randomly. It is evident that the accuracy of MLE decreases as \(F_0\) becomes larger and the approximation worsens, while MKSDE remains accurate for all values of \(F_0\). This result demonstrates the accuracy and stability of MKSDE over MLE.

\paragraph{Object Orientation Estimates:} We now describe an experiment where, samples are first drawn from a vMF on $\SO(3)$, with known parameter $F$.  The vMF can then be sampled to generate distinct orientations, $X_i \in \SO(3)$ which are applied to objects in the ModelNet10\cite{wu20153d} database of CAD (computer aided design)  models and sample models are depicted in figure \ref{CADmodels}. Now, the goal of this experiment is to estimate the parameter $F$ given the samples using MKSDE and MLE yielding $\hat{F}_{\ksd}$ and $\hat{F}_{\mle}$ respectively, which can then be compared. Note that once the parameter $F$ is estimated from the samples, we can then compute an estimate of the mode of the distribution (see \cite{EggertMVA97}) which we will denote here by $\hat{X}$, given by $\hat{X} = U \text{diag}(1,1,\det(UV))V^t$, where, $U$ and $V$ are orthogonal matrices obtained from the singular value decomposition (SVD) of $\hat{F}$. We can then easily compare the ground truth object orientation of the CAD models in the data base to estimated object orientations (that will correspond to the estimated mode of the vMF). Before presenting these object orientation estimates, we present comparisons between mKSDE and MLE estimates of the parameter $F$.

%%We simulate from a vMF on \(\SO(3)\) with ground truth \(F_0\), and compare the MKSDE \(\hat{F}_{\ksd}\) to the MLE \(\hat{F}_{\mle}\) obtained via the above mentioned approximation. 

\begin{figure*}[!t]
\centering
\subfloat[\(F_0 = 0.1*I\)]{\includegraphics[width =0.3\linewidth]{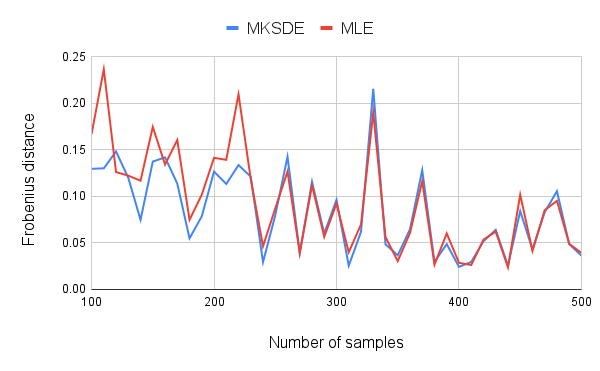}}
\hfil
\subfloat[\(F_0 = 0.5*I\)]{\includegraphics[width =0.3\linewidth]{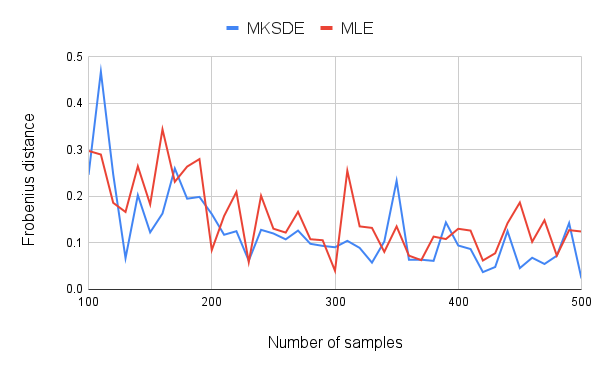}}
\hfil
\subfloat[\(F_0 = I\)]{\includegraphics[width =0.3\linewidth]
{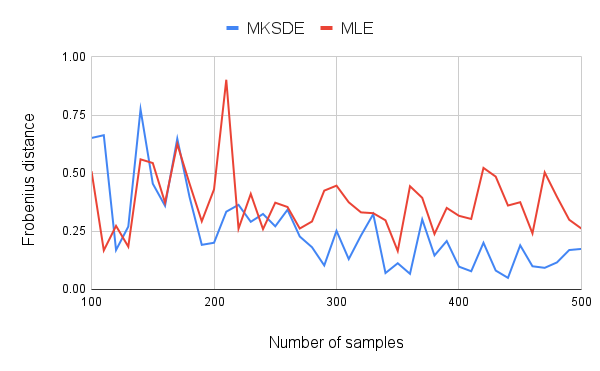}}
\hfil
\subfloat[\(F_0 = 5*I\)]{\includegraphics[width =0.3\linewidth]
{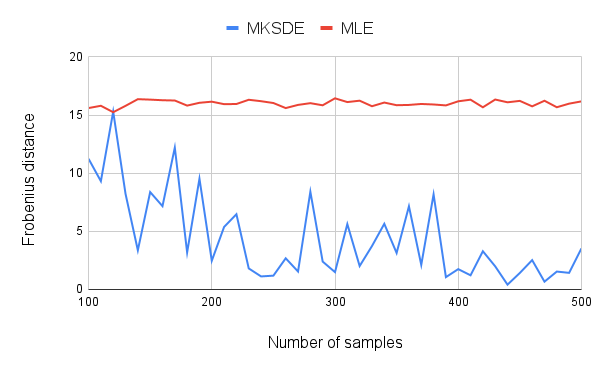}}
\hfil
\subfloat[\(F_0=\text{diag}(0.1,0.2,0.3)\)]{\includegraphics[width =0.3\linewidth]
{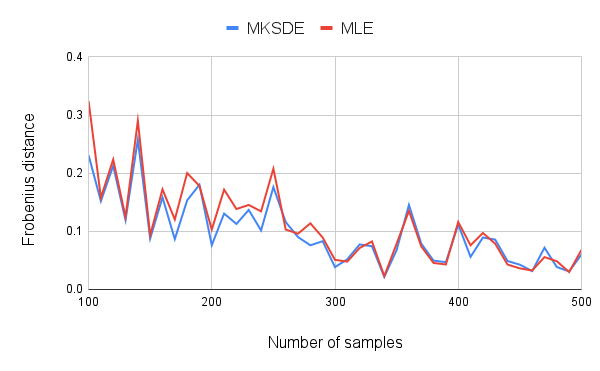}}
\hfil
\subfloat[\(F_0 =\) random]{\includegraphics[width =0.3\linewidth]
{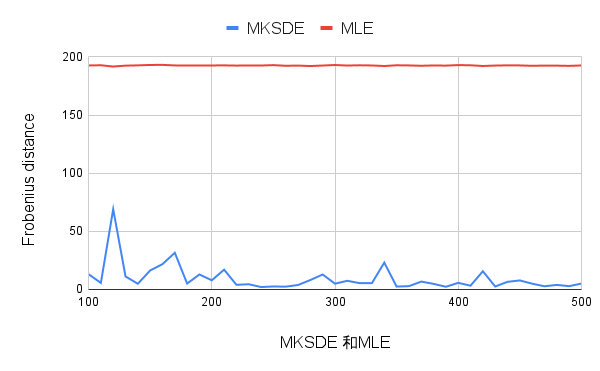}}
\centering
\caption{$F$-norms between estimators and ground truth }
\label{MKSDE&MLE}
\label{fig:mle}
\end{figure*}

% \begin{figure}[!t]
% \centering
% \subfloat[\(F_0 = I\)]{
%     \begin{minipage}[t]{0.15\textwidth}
%     \includegraphics[width = \textwidth]{Image/bathtub01.png}
%     \end{minipage}
%     \begin{minipage}[t]{0.15\textwidth}
%     \includegraphics[width = \textwidth]{Image/bathtub02.png}
%     \end{minipage}
%     \begin{minipage}[t]{0.15\textwidth}
%     \includegraphics[width = \textwidth]{Image/bathtub02.png}
%     \end{minipage}
% }

% \subfloat[\(F_0 = 5*I\)]{
%     \begin{minipage}[t]{0.15\textwidth}
%     \includegraphics[width = \textwidth]{Image/chair01.png}
%     \end{minipage}
%     \begin{minipage}[t]{0.15\textwidth}
%     \includegraphics[width = \textwidth]{Image/chair02.png}
%     \end{minipage}
%     \begin{minipage}[t]{0.15\textwidth}
%     \includegraphics[width = \textwidth]{Image/chair03.png}
%     \end{minipage}
% }

% \subfloat[\(F_0=\text{diag}(0.1,0.2,0.3)\)]{
%     \begin{minipage}[t]{0.15\textwidth}
%     \includegraphics[width = \textwidth]{Image/bed01.png}
%     \end{minipage}
%     \begin{minipage}[t]{0.15\textwidth}
%     \includegraphics[width = \textwidth]{Image/bed02.png}
%     \end{minipage}
%     \begin{minipage}[t]{0.15\textwidth}
%     \includegraphics[width = \textwidth]{Image/bed03.png}
%     \end{minipage}
% }
% \caption{Sample a rotation matrix from the matrix Fisher distribution with a varying parameter $F_0$, and apply the rotations to objects in ModelNet10. footnote{https://modelnet.cs.princeton.edu/}}
% \label{CADmodels}
% \end{figure}

\begin{figure}[!t]
\centering
\subfloat[\(F_0 = I\)]{
    \begin{minipage}[t]{0.15\textwidth}
    \includegraphics[width = \textwidth]{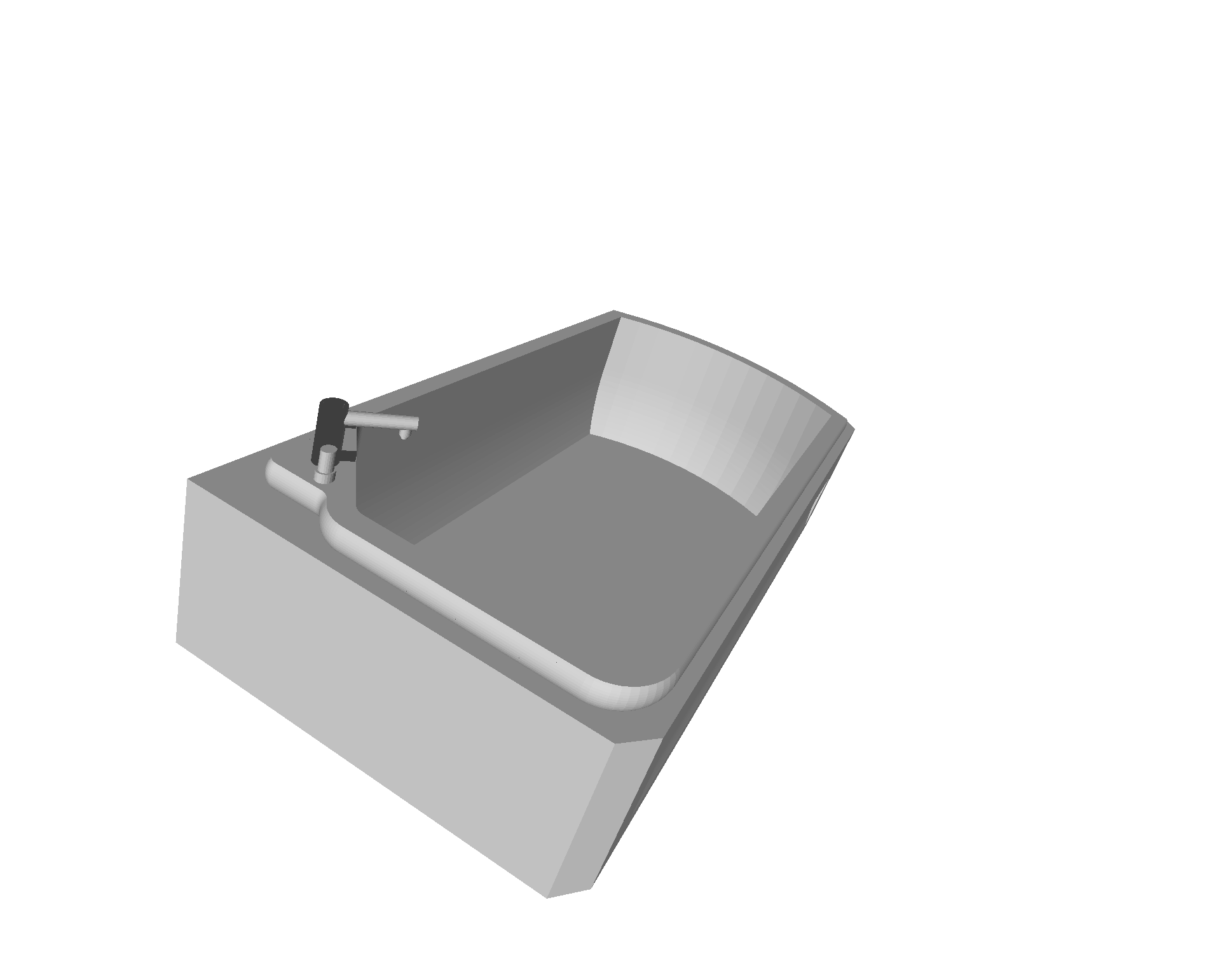}
    \end{minipage}
    \begin{minipage}[t]{0.15\textwidth}
    \includegraphics[width = \textwidth]{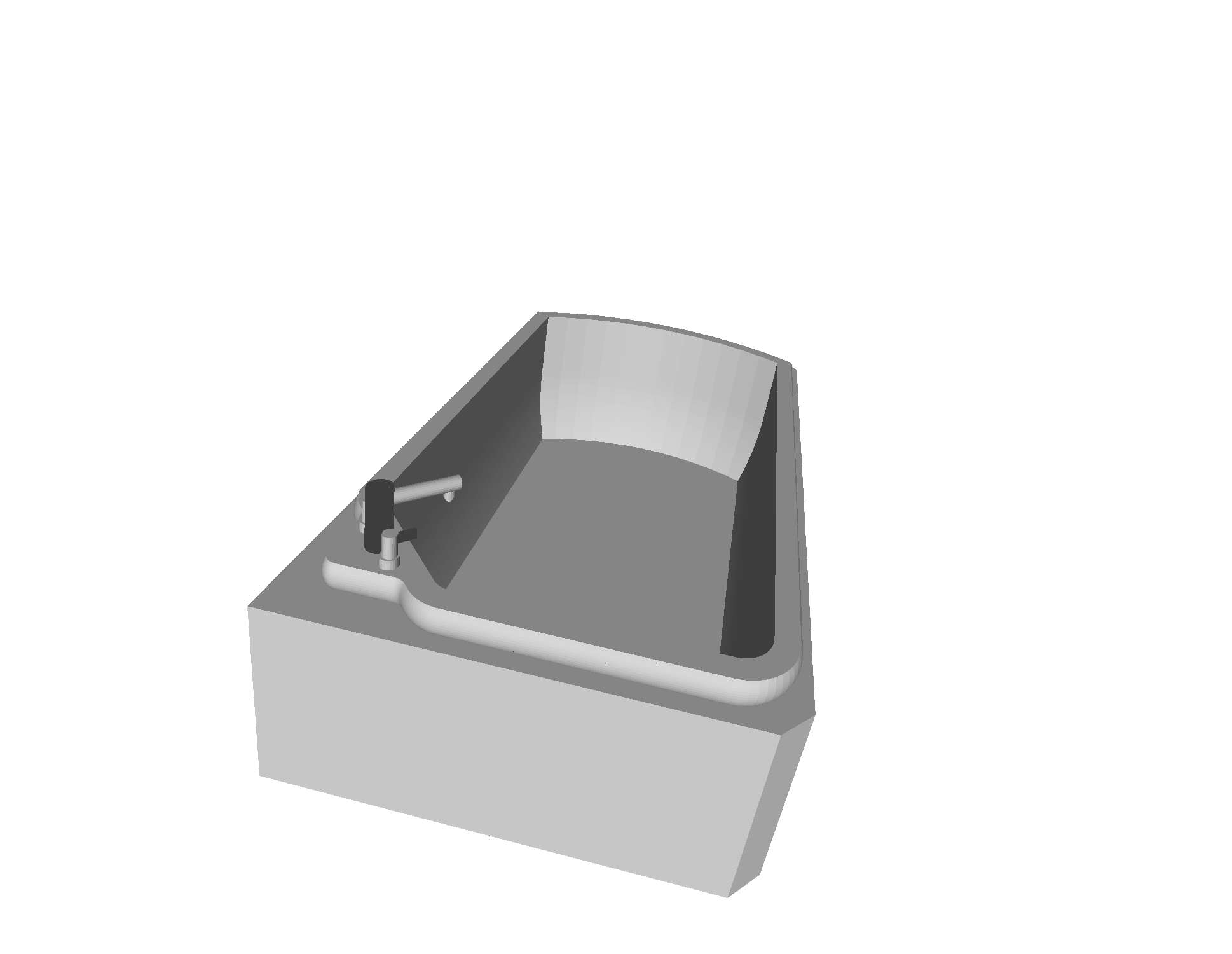}
    \end{minipage}
    \begin{minipage}[t]{0.15\textwidth}
    \includegraphics[width = \textwidth]{Image/bathtub02.png}
    \end{minipage}
}

\subfloat[\(F_0 = 5*I\)]{
    \begin{minipage}[t]{0.15\textwidth}
    \includegraphics[width = \textwidth]{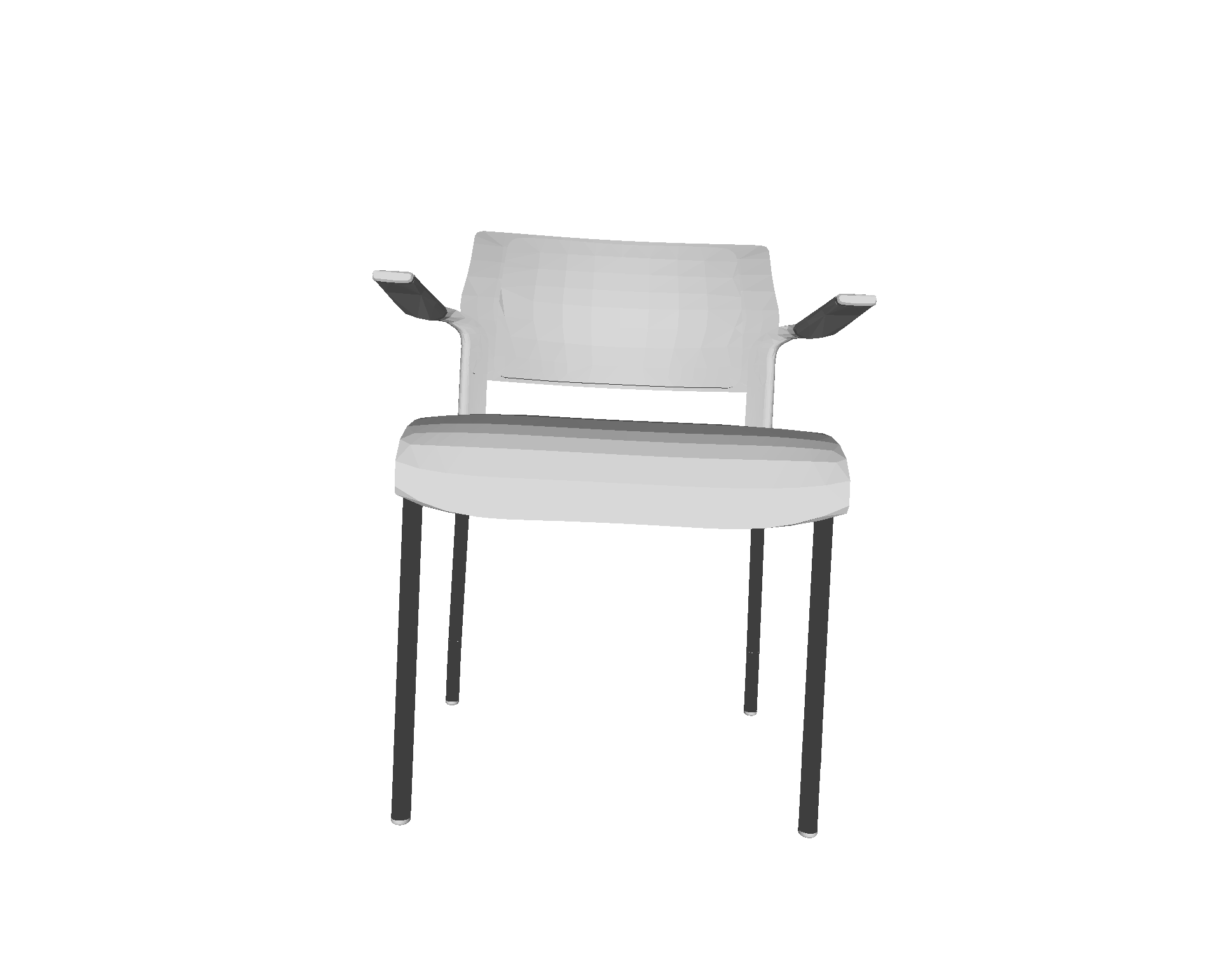}
    \end{minipage}
    \begin{minipage}[t]{0.15\textwidth}
    \includegraphics[width = \textwidth]{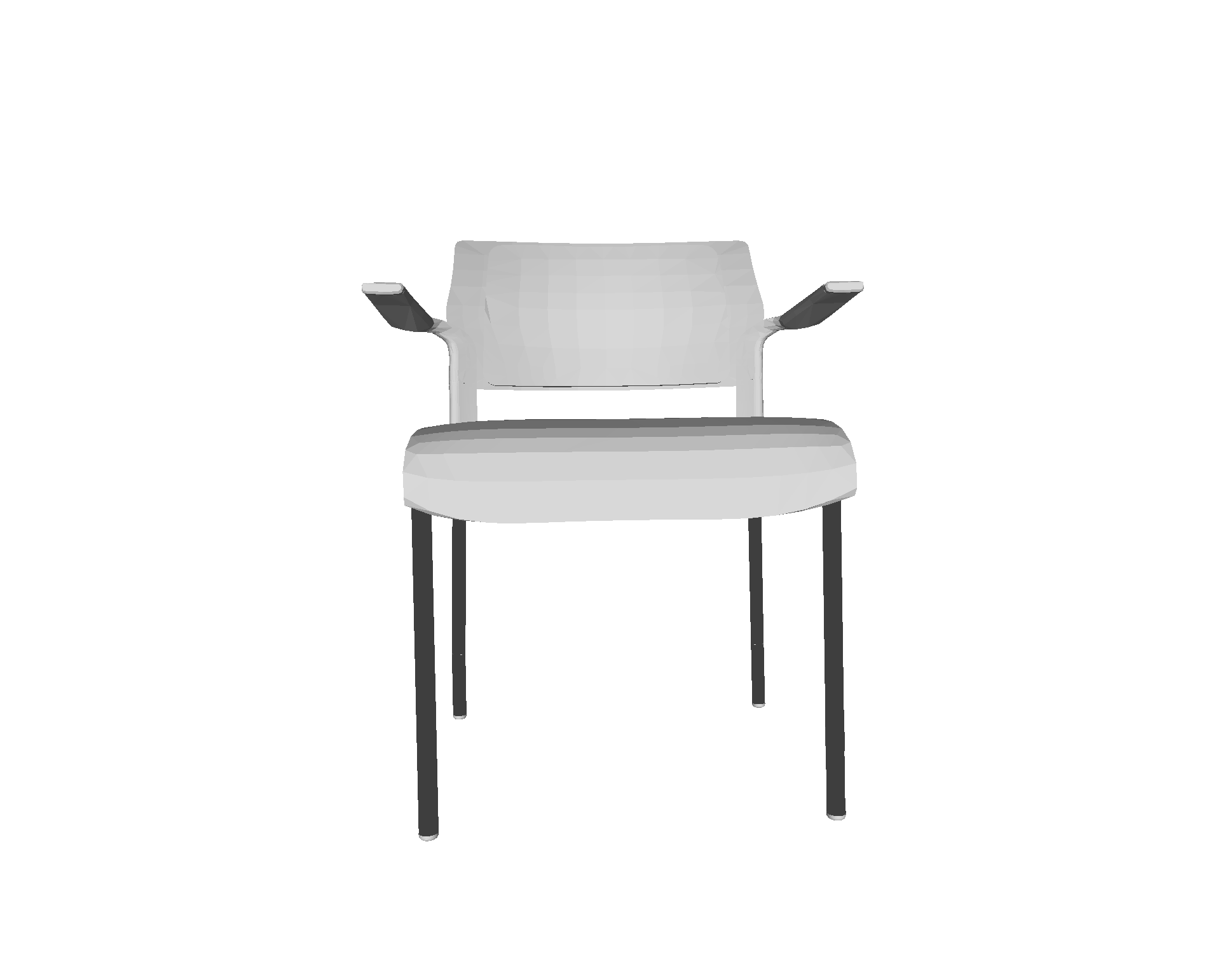}
    \end{minipage}
    \begin{minipage}[t]{0.15\textwidth}
    \includegraphics[width = \textwidth]{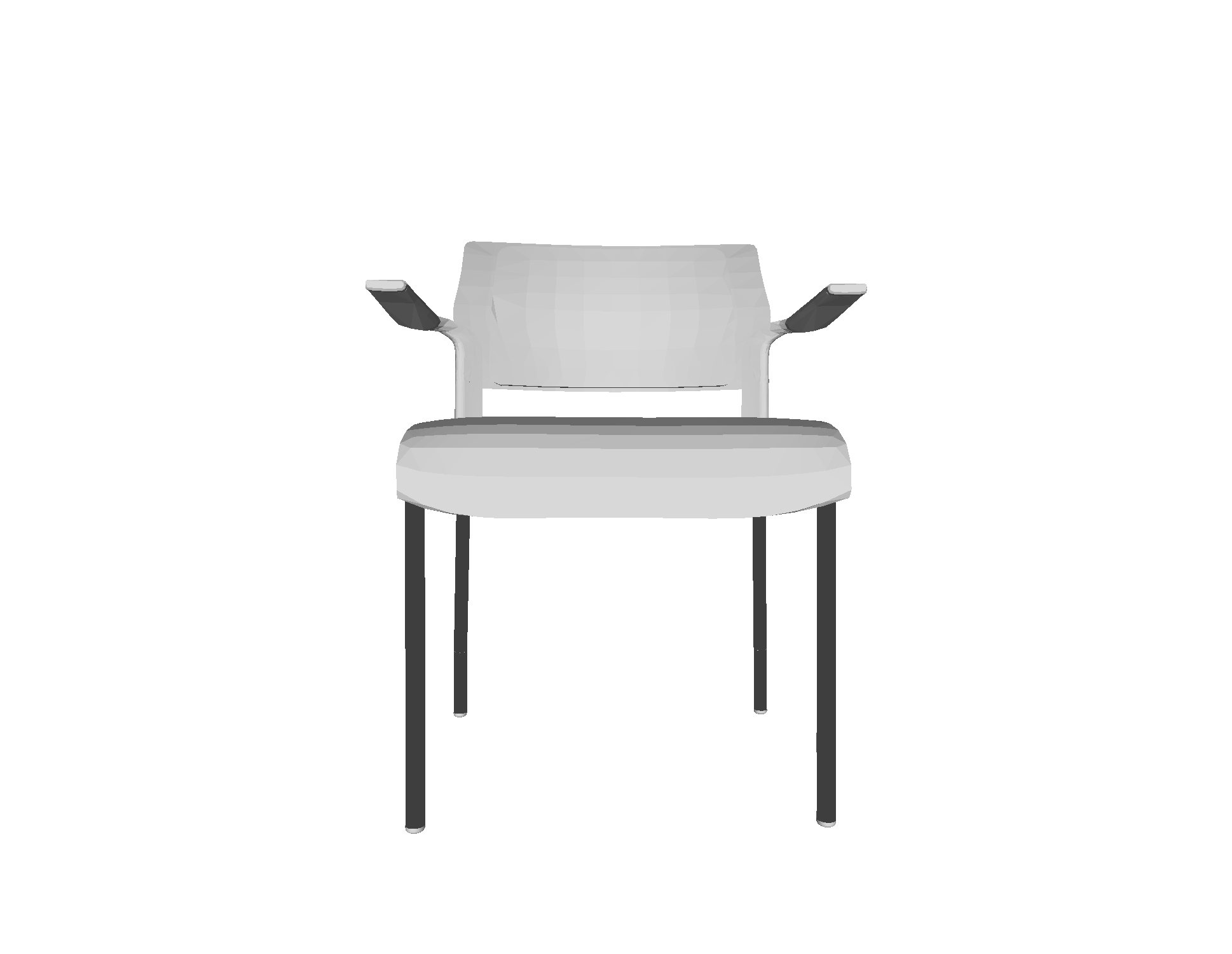}
    \end{minipage}
}

\subfloat[\(F_0=\text{diag}(0.1,0.2,0.3)\)]{
    \begin{minipage}[t]{0.15\textwidth}
    \includegraphics[width = \textwidth]{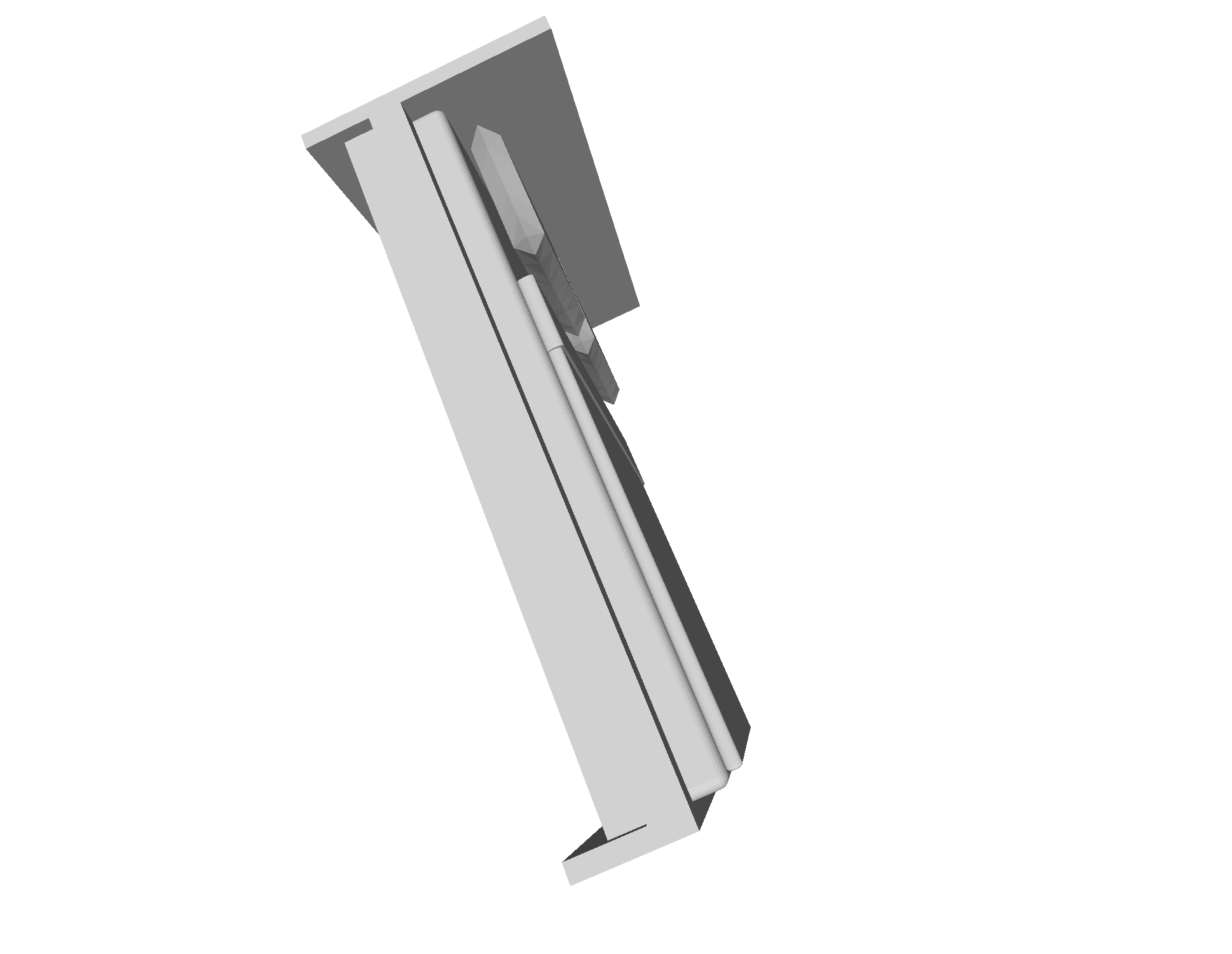}
    \end{minipage}
    \begin{minipage}[t]{0.15\textwidth}
    \includegraphics[width = \textwidth]{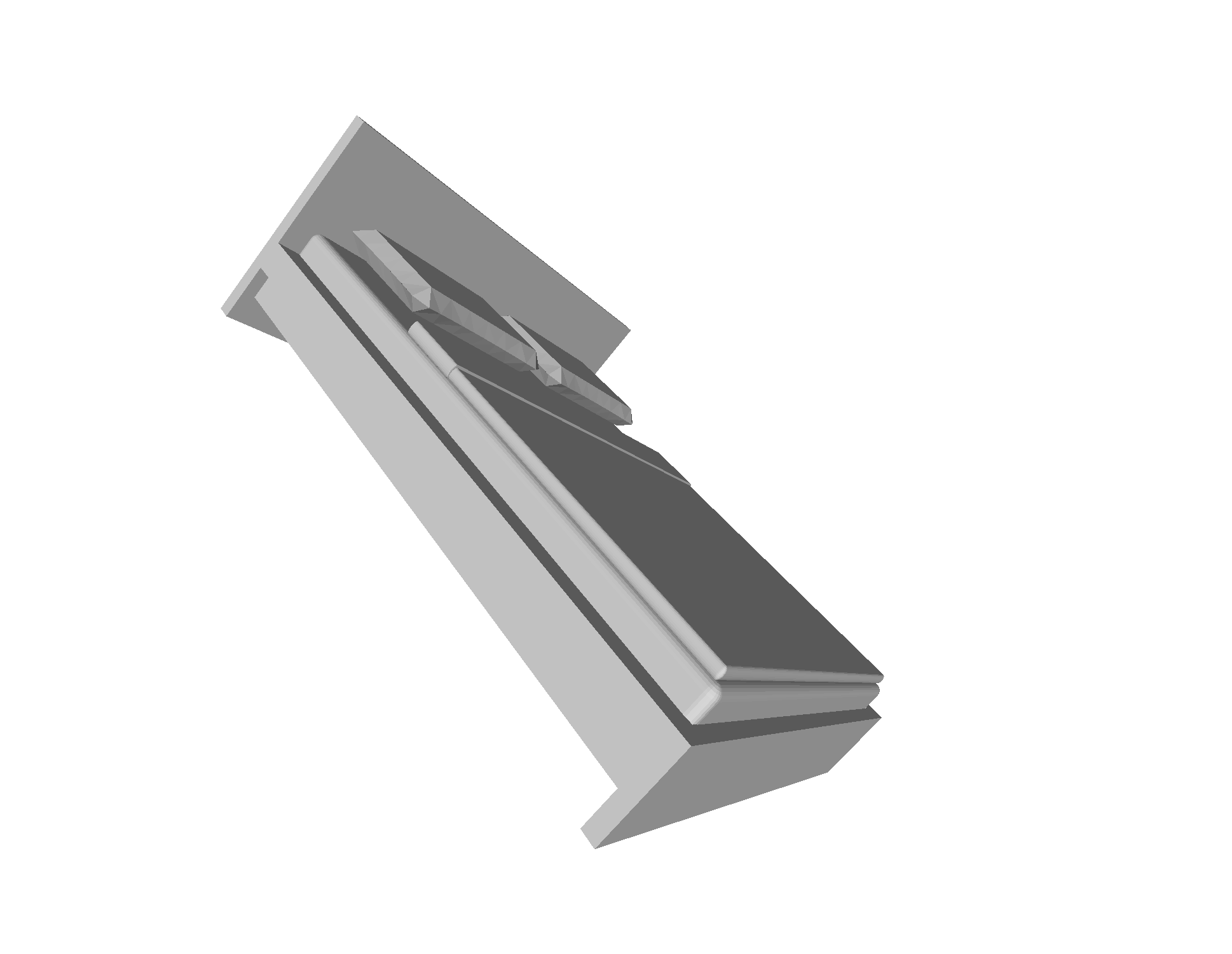}
    \end{minipage}
    \begin{minipage}[t]{0.15\textwidth}
    \includegraphics[width = \textwidth]{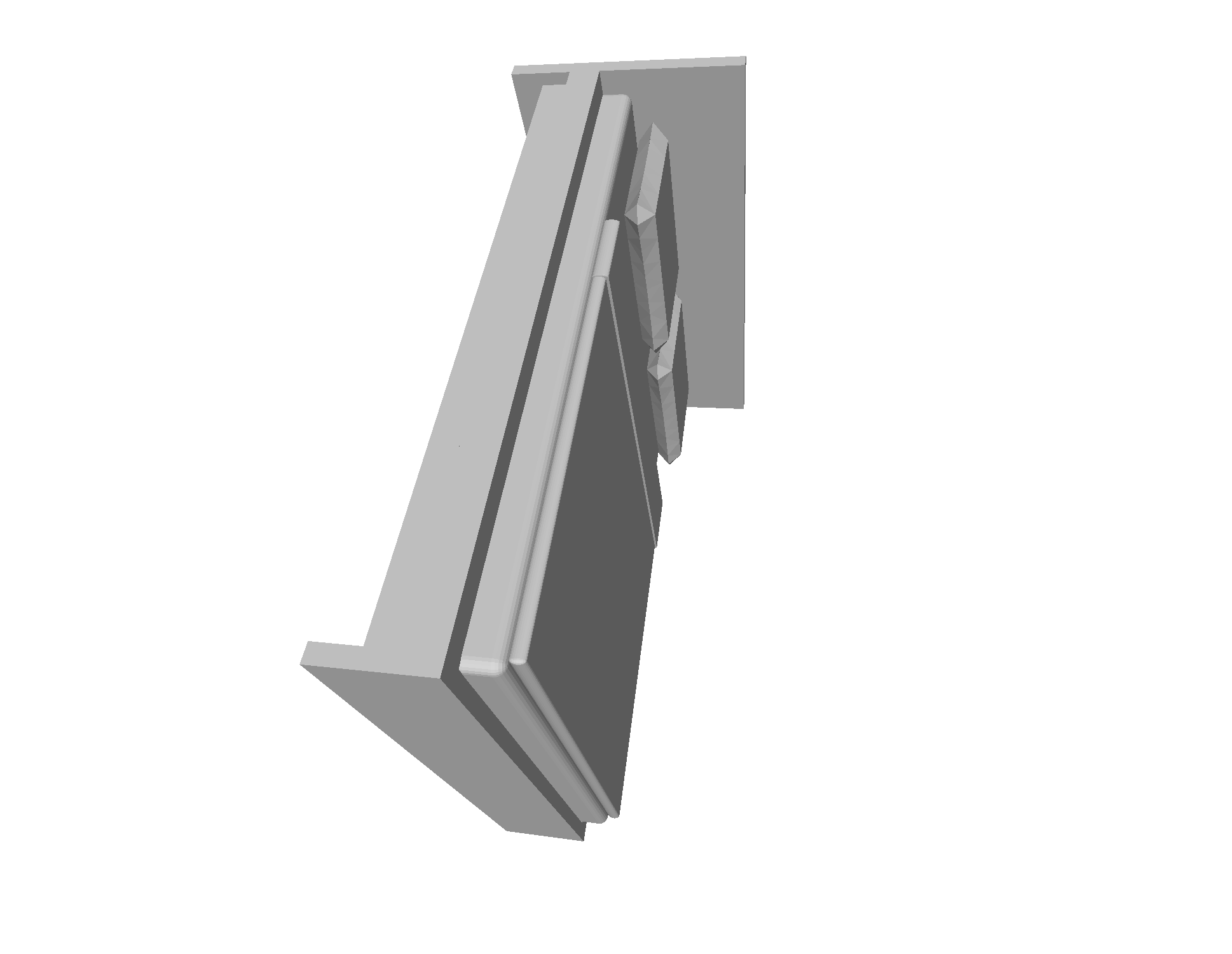}
    \end{minipage}
}
\caption{Sampled rotation matrices from a matrix Fisher distribution with a varying parameter $F_0$, applied to example objects in the ModelNet10 database.}

\label{CADmodels}
\end{figure}
\begin{table}[ht]
\centering
\caption{Geodesic distance between the mode of estimated $F$ ($\hat{F}_{\ksd}$ and $\hat{F}_{\mle}$) and the ground truth orientation of the samples shown in Fig ~\ref{CADmodels}}
\begin{tabular}{clrrr}
\hline
\multicolumn{1}{l}{}                                    &       & left    & middle    & right    \\ \hline
\multicolumn{1}{c|}{\multirow{2}{*}{(a) $F_0=I$}}       & MKSDE & \textbf{1.53} & \textbf{2.82} & \textbf{3.01} \\
\multicolumn{1}{c|}{}                                   & MLE   & 1.54 & 3.02 & 3.07 \\ \hline
\multicolumn{1}{c|}{\multirow{2}{*}{(b) $F_0= 5 * I$}}       & MKSDE & \textbf{0.93} & \textbf{1.55} & \textbf{1.95} \\
\multicolumn{1}{c|}{}                                   & MLE   & 2.66 & 2.91 & 3.01 \\ \hline
\multicolumn{1}{c|}{\multirow{2}{*}{(c) $F_0 = \text{diag}(0.1,0.2,0.3)$}} & MKSDE & \textbf{1.62} & \textbf{1.65} & \textbf{1.64} \\
\multicolumn{1}{c|}{}                                   & MLE   & 1.76 & 1.79 & 1.78 \\ \hline
\end{tabular}
\label{distance}
\end{table}

We now consider the estimated $F$ ($\hat{F}_{\ksd}$ and $\hat{F}_{\mle}$) obtained from a sample size of 500. In Table ~\ref{distance}, we report the geodesic distance between the mode of ground truth orientation for the samples shown in Fig ~\ref{CADmodels} and the mode of estimated $F$. From the table, it is evident that the gedesic distance of MKSDE estimates are consistently smaller than those for MLE indicating a superior performance of MKSDE.

\subsection{MKSDE goodness of fit test}

In this experiment, we measure the difference between a specific Cayley distribution to the vMF family. The Cayley distribution \cite{leon2006statistical} has the density \(p(X|M)\propto \det(I+XM^\top)^\kappa \) with parameters \(M\in\SO(3)\) and \(\kappa>0\). As the vMF family is rotationally symmetric, the parameter \(M\) does not affect the dissimilarity between a specific Cayley distribution and the vMF family. We used the R package \cite{bryan2014rotations} to generate \(n=500\) samples from the Cayley distribution for varying \(\kappa\), and perform the goodness of fit test in \S\ref{Goodness-of-Fit} with different levels of significance \(\beta\). 

Table \ref{Goodness-of-Fit} depicts the \((1-\beta)\)-quantile and the statistic \(n \wksd^2_n(\hat{\theta}_n)\). As discussed in \cite{leon2006statistical}, the Cayley distribution resembles a uniform distribution for small \(\kappa\), and a spiky local Gaussian distribution for large \(\kappa\). It approximately belongs to the vMF class in both cases, but differs from vMF for \(\kappa\) in-between. This coincides with the results in table \ref{Goodness-of-Fit}. 

\begin{table}[ht]
\centering
\caption{MKSDE goodness of fit}
\label{Goodness-of-Fit}
\begin{tabular}{crrrrr}
\hline
\(\kappa\) & 0.2      & 0.5      & 1.0        & 1.5      & 2.0 \\
\hline
\(n\cdot \wksd^2_n(\hat{\theta}_n)\)   & 68.65 & 90.25 & 99.99 & 99.29 & 104.23 \\
\hline
\(\beta = 0.01\)     & 83.98 & 94.23 & 108.18 & 116.88 & 139.76 \\
\(\beta = 0.05\)     & 77.62  & 83.21 & 96.99 & 104.18 & 114.47 \\
\(\beta = 0.10\)     & 73.22 & 77.79  & 90.71 & 96.46 & 107.12\\\hline
\end{tabular}
\end{table}

\section{Conclusions}\label{conclusion}

In this paper, we presented a novel Stein's operator defined on Lie Groups leading to a kernel Stein discrepancy (KSD) which is a normalization-free loss function. We presented several theoretical results characterizing the properties of this new KSD on Lie groups and the MKSDE. We presented new theorems on MKSDE being strongly consistent and asymptotically normal, and a closed form expression of the MKSDE for the exponential family and in particular the vMF distribution and the Riemannian normal distribution on \(\SO(N)\). Furthermore, we presented two algorithms, namely MKSDE goodness of fit and KSD-EM, for measuring the quality of model fitting and distributional parameter estimation with latent variables respectively. Finally, we presented several experiments demonstrating the advantages of MKSDE over MLE. Our future work will focus on exploring the practical implications of the choice of $C_0$-universal kernels in characterizing the Stein class.

% \textcolor{blue}{Our properties are about the global minimizer}
% \textcolor{red}{Local minimizers are for non-convex loss functions and what does it mean to say asymptotic properties of such minimizers? Do you mean an iterative optimization algorithm will converge to a local optimum?  }
% In this paper, we did not delve in to the properties of local minimizers of KSD, while in many practical situations, a local minimizer is the best one can hope for. In our future work, we will develop the asymptotic theory of the local minimizers of KSD and investigate efficient algorithms for realizing the theoretically established global minimizers in theorem \ref{MKSDEconsistency}.

\section*{Acknowledgments}This research was in part funded by the NIH NINDS and NIA grant RO1NS121099 to Vemuri.

\appendix
\label{Appendix}
\section{Proof of theorems}\label{Proofs}

\subsection{Proof of theorem \ref{KSDCharacterization} }\label{Proof-3.4}

\begin{theorem*}[Characterization] Suppose \(k\) is \(C_0\)-universal and \(\sqrt{k_p(x,x)}\) is integrable with respect to locally Lipschitz densities \(p\) and \(q\). Suppose further \(D^l \log(p/q)\) is \(p\) and \(q\)-integrable for all \(l\). Then \(p=q\iff\ksd(p,q)=0 \).
\end{theorem*}

\begin{proof}
We begin with the forward implication, "$\Rightarrow$". By \cite[Cor. 4.36]{steinwart2008support}, we have $|\mathcal{A}_p f(x)|\leq \Vert f\Vert_{\mathcal{H}_k}\cdot \sqrt{k_p(x,x)}$ for $f\in\mathcal{H}_k$. Moreover, the RKHS of a $C_0$-universal kernel contains only $C_0$, thus bounded functions, see \cite[Prop. 2.3.2]{carmeli2010vector}. Therefore, for all $f\in\mathcal{H}_k$,  $f,\mathcal{A}_p f$ are both $q$-integrable. Let $\Div$ be the divergence operator with respect to the left-invariant Riemannian metric. Recall that  $\Div(D^l)=D^l\Delta$, thus $\Div( p f_l D^l ) = p\cdot f_l D^l\Delta+ p\cdot D^l f_l+f_l\cdot D^l p=\mathcal{A}^l_p f_l\cdot p $. By generalized Stokes's theorem
\cite{gaffney1954special}, if an integrable vector field on a complete manifold has integrable divergence, then its divergence integrates to $0$. Therefore, $\mathbb{E}_p[\mathcal{A}_p f]=\int_G \mathcal{A}_p f\cdot p d\mu=\int_G \Div(\sum_{l=1}^d p f_l D^l) d\mu = 0 $. For readers unfamiliar with the divergence operator or Stokes's theorem, this proof maybe be interpreted as a general version of integration by parts on manifolds.

Now, we address the reverse implication, "$\Leftarrow$". Note that $\mathbb{E}_q[\mathcal{A}_p f]=\mathbb{E}_q[\mathcal{A}_p f-\mathcal{A}_q f]=\sum_l \mathbb{E}_q[f_l  D^l\log(p/q)]  $, since the operator $\mathcal{A}_q$ corresponding to $q$ also has the Stein's identity. If $k$ is $C_0$-universal, then $\mathcal{H}_k$ is dense in the space of continuous functions vanishing at infinity \cite[Thm. 4.1.1]{carmeli2010vector}. Since $\mathbb{E}_q[f^l D^l\log(p/q)]=0$ for all $f\in \mathcal{H}^d_k$, we have $D^l \log(p/q)=0$ for all $l$. As we assume $G$ is connected, $\log(p/q)$ is constant and thus $p=q$.
\end{proof}

\subsection{Proof of theorem \ref{KSDasymptotic-stronger}}\label{Proof-4.2}

\begin{theorem*} Suppose \(V_{(\cdot)}(\cdot,\cdot)\) is jointly continuous and satisfies that \(\sup_{\theta\in K} V_\theta(x,x)\) is \(w\)-integrable for any compact \(K\subset\Theta\), then \(\wksd^2_n(\theta)\to \ksd^2(\theta) \) compactly almost surely, i.e., for any compact \(K\),
\[ \wksd^2_n(\theta) \to \ksd^2(\theta) \text{ uniformly on }K, \quad\text{almost surely}. \]
As a corollary, if \(\Theta\) is locally compact, \(\wksd_n\) and \(\ksd\) are all continuous on \(\Theta\).
\end{theorem*}

We prove a lemma first.

\begin{lemma}\label{Uniform-LLN}
Suppose $f(\theta,x)$ is continuous on $\Theta\times G$ and $\sup_{\theta\in K} |f(\theta,x)|$ is $P$-integrable for all compact $K\subset\Theta$. %Suppose $x_i$ are samples from $P$. 
Then
\[n^{-1}\sum_{i=1}^n f(\theta,X_i)\to F(\theta):= \mathbb{E}[f(\theta,X_1)],\quad X_i \overset{ \text{\tiny i.i.d}}{\sim} P.
\]
compactly $P$-almost surely on $\Theta$ and thus $F(\theta)$ is continuous.
\end{lemma}

\begin{proof}The lemma relies on a classical results \cite[Thm. 1]{rubin1956uniform} regarding the uniform convergence of random functions. It suffices to check the condition of equi-continuity. Take a sequence of compact sets $G_n$ such that $G=\bigcup G_n$, then the join continuity of $f(\cdot,\cdot)$ will implies the joint uniform continuity of $f(\theta,x)$ on $K\times G_n$, which further implies equi-continuity. 
\end{proof}

Next we prove the theorem:

\begin{proof}
Note that, 
\[
\begin{aligned}
\wksd_n(\theta):= \frac{1}{n^2} \sum_{i=1}^n V_\theta(x_i,x_i)+\frac{n-1}{n}\frac{1}{n(n-1)}\sum_{i\neq j} V_\theta(x_i,x_j).
\end{aligned}
\]

The first term tends to $0$ compactly almost surely by lemma \ref{Uniform-LLN}. For the second term, we apply a result \cite[Thm. 1]{yeo2001uniform} regarding the uniform convergence of $U$-statistic to conclude that
\[
\frac{1}{n(n-1)}\sum_{i\neq j} V_\theta(x_i,x_j)\to \mathbb{E}_{X,Y\sim p}[V_\theta(X,Y)]
\]
compactly almost surely. These two statements together will imply the theorem. It suffices to test the conditions (i), (ii) and (iii) in \cite[Thm. 1]{yeo2001uniform}.

For condition (i), note that $V_\theta(\cdot,\cdot)$ is also a positive definite bivariate function, thus
\[ \sup_{\theta\in K} |V_\theta(x,y)|\leq \sup_{\theta\in K}V_\theta(x,x)^{\nicefrac{1}{2}}\cdot \sup_{\theta\in K}V_\theta(y,y)^{\nicefrac{1}{2}},\]
Recall that the $w$-integrability of $\sup_{\theta\in K}V_\theta(x,x)$ implies the $w$-integrability of $\sup_{\theta\in K}V_\theta(x,x)^{\nicefrac{1}{2}}$ by Jensen's inequality. Therefore, $\sup_{\theta\in K}|V_\theta(x,y)|$ is $w\times w$-integrable. Therefore, the condition (i) holds.

For condition (ii), we take a sequence of compact subset $G_n$ such that $G=\bigcup G_n$.

For condition (iii), note that $K\times G_n\times G_n$ is compact, thus $V_\theta(x,y)$ is uniformly continuous on $K\times G_n\times G_n$, thus equi-continuous in $\theta$ for $(x,y)\in G_n\times G_n$. Next we show the equi-continuity of $\mathbb{E}_{Y\sim P}[V_\theta(x,Y)]$ on $K\times G_n$. Given a fixed $(\theta_0,x_0)\in K\times G_n$, we choose a compact neighborhood $W\subset G$ of $x_0$. Since $V_\theta(x,y)$ is jointly continuous, $V_\theta(x,x)$ is bounded on the compact set $K\times W$ and we denote $C:=\sup_{\theta\in K, x\in W} [V_\theta(x,x)]^{\nicefrac{1}{2}}<+\infty$. For any $\theta\in K$, $x\in W$ and $y\in G$, we have
\[ |V_\theta(x,y)|\leq \sup_{\theta\in K, x\in W} [V_\theta(x,x)]^{\nicefrac{1}{2}}\cdot \sup_{\theta\in K} [V_\theta(y,y)]^{\nicefrac{1}{2}}= C\cdot \sup_{\theta\in K} [V_\theta(y,y)]^{\nicefrac{1}{2}} \]
For any $(\theta',x')\in K\times W$ that tends to $(\theta_0,x_0)$,
we have $|V_{\theta'}(x',Y)-V_{\theta_0}(x_0,Y)|\leq 2 C \sup_{\theta\in K} [V_\theta(y,y)]^{\nicefrac{1}{2}}$. Since the right hand side is a $P$-integrable function of $y$ as the condition states, we have
\[
|\mathbb{E}_{Y\sim P}[V_{\theta'}(x',Y)]-\mathbb{E}_{Y\sim P}[V_{\theta_0}(x_0,Y)]|\leq \mathbb{E}_{Y\sim P}|V_{\theta'}(x',Y)-V_{\theta_0}(x_0,Y)|\to 0, \quad \text{as } (\theta',x')\to(\theta_0,x_0).
\]
Therefore, $\mathbb{E}_{Y\sim w}[V_{(\cdot)}(\cdot,Y)]$ restricted on $K\times G$ is jointly continuous at $(\theta_0,x_0)$, and thus continuous on $K\times G$ by the arbitrariness of $\theta_0$ and $x_0$. It is thus uniformly continuous on $K\times G_n$ and hence the condition of equi-continuity holds.

The continuity of $\wksd_n$ is straightforward since it is a finite summation of continuous functions. Since $\ksd$ is the uniform limit of $\wksd_n$ on any compact set, $\ksd$ is continuous on any compact set. Therefore, $\ksd$ is continuous as continuity is a local property. This completes the proof.
\end{proof}

\subsection{Proof of theorem \ref{MKSDEconsistency}}\label{Proof-4.3}

\begin{theorem*}[Strong consistency]Suppose the conditions in theorem 4.2 hold and \(\Theta=\Theta_1\times\Theta_2\) such that \(\Theta_1\) is compact, \(\Theta_2\) is convex, and for each fixed \(\theta_1\in\Theta_1\), \(\wksd_n(\theta_1,\cdot)\) is convex on \(\Theta_2\) and \(\ksd(\theta_1,\cdot)\) attains minimum value on a non-empty and compact set \(\Tilde{\Theta}_0(\theta_1)\subset \text{int}(\Theta_2)\). Then \(\Theta_0\), \(\widehat{\Theta}_n\) are non-empty for large \(n\) and \(\sup_{\theta\in \widehat{\Theta}_n} d(\theta,\Theta_0)\to 0\) almost surely.
\end{theorem*}

We now establish some lemmas that will be needed for the proof.

\begin{lemma}\label{Convex} Suppose $f$ is convex and $C$ is a bounded open set. If there exists $x_0\in C$ such that $f(x_0)<\inf_{x\in\partial C} f(x)$, then $\inf_{x\notin C} f(x)= \inf_{x\in \partial C} f$ and the set of global minimizers of $f$ is non-empty and contained in $C$.
\end{lemma}
\begin{proof} For all $x'\notin C$, there exists $0\leq\beta< 1$ such that $\beta x_0+(1-\beta)x'\in \partial C$, and thus 
\[
\begin{aligned}
\beta f(x_0)+(1-\beta) f(x')\geq f(\beta x_0+(1-\beta)x') >f(x_0)\Longrightarrow f(x')>f(x_0).  
\end{aligned}
\]
Furthermore,
\[
\begin{aligned}
f(x')\geq \beta f(x_0)+(1-\beta)f(x') \geq f(\beta x_0+(1-\beta)x')\geq \inf_{x\in \partial C} f(x). 
\end{aligned}
\]
This concludes the proof.
\end{proof}

Note that $\Theta$ is locally compact, thus $\wksd_n$ and $\ksd$ are continuous on $\Theta$ by theorem 4.3. Furthermore, as the pointwise limit preserves the convexity, $\ksd(\theta_1,\cdot)$ is also convex in $\theta_2$ for all $\theta_1\in\Theta_1$. Let $m^*(\theta_1):=\inf_{\theta_2\in\Theta_2}\ksd(\theta_1,\theta_2)$, and denote by $\Tilde{\Theta}_0(\theta_1)$ the set of minimizers of $\ksd(\theta_1,\cdot)$. Then we establish the next lemma.

\begin{lemma}\label{KSDminSet} For each $\Bar{\theta}_1\in\Theta_1$, there exists a neighborhood $V_1$ of $\Bar{\theta}_1$ and a bounded open set $V_2\subset\Theta_2$ such that for any $\theta_1\in V_1$, $\Tilde{\Theta}_0(\theta_1)\subset V_2$.
\end{lemma}

\begin{proof} Take a bounded $V_2$ such that $\Tilde{\Theta}_0(\Bar{\theta_1})\subset V_2$, then since $\ksd(\Bar{\theta}_1,\cdot)-m^*(\Bar{\theta}_1)$ is continuous and positive on $\partial V_2$, let $\epsilon:=\inf_{\theta_2\in\partial V_2}\ksd(\Bar{\theta}_1,\theta_2)-m^*(\Bar{\theta}_1)>0$. Note that $\{\Bar{\theta}_1\}\times\partial V_2$ is a subset of the open set $\{\ksd(\cdot,\cdot)>\epsilon/2+m^*(\Bar{\theta}_1)\} $, thus we apply tube lemma \cite[Lem. 26.8]{munkres2000topology} and get an open neighborhood $V_1$ of $\Bar{\theta}_1$ such that $V'_1\times \partial V_2\subset \{\ksd(\cdot,\cdot)>\epsilon/2+m^*(\Bar{\theta}_1)\}$, which implies that $\inf_{V'_1\times \partial V_2}\ksd > m^*(\Bar{\theta}_1)+\epsilon/2$. However, we choose a  $\Tilde{\theta}_2\in \Tilde{\Theta}_1(\Bar{\theta}_1)\subset V_2$, there exists a neighborhood $V''_1$ of $\Bar{\theta}_1$ such that $\sup_{\theta_1\in V''_1} \ksd(\theta_1,\Tilde{\theta}_2)<m^*+\epsilon/2$, as $\ksd$ is continuous at $(\Bar{\theta}_1,\Tilde{\theta}_2)$. Therefore, for any $\theta_1\in V_1:=V'_1\cap V''$, we have $\ksd(\theta_1, \Tilde{\theta}_2 )< \inf_{\theta_2\in\partial V_2} \ksd(\theta_1,\theta) $. We now apply lemma \ref{Convex}, resulting in the minimizer set $\Tilde{\Theta}_0(\theta_1)$ is contained in $V_2$.
\end{proof}

Now we are ready to prove the theorem. The proof is presented in several stages each of which is required to get to the desired result.

\begin{proof}
1. First we show that $\Theta_0$ is non-empty.

Let $m^*_0:=\inf_{\theta_1\in\Theta_1}m^*(\theta_1)=\inf_{\theta\in\Theta}\ksd(\theta)$. Take a sequence $\theta_{1,n}\in \Theta_1$ such that $m^*(\theta_{1,n})\downarrow m^*_0 $. As $\Theta_1$ is compact, $\theta_{1,n}$ has a convergent sub-sequence, still denoted as $\theta_{1,n}$, whose limit point is denoted as $\theta_{1,0}$. We apply lemma \ref{KSDminSet} to $\Tilde{\Theta}_0(\theta_{1,0})$, thus there exists a compact neighborhood $V_1\times V_2$ of $\{\theta_{1,0}\}\times\Tilde{\Theta}_0(\theta_{1,0}) $ such that for all $\theta_1\in V_1$, $\Tilde{\Theta}_0(\theta_1)\subset V_2$. Since $\ksd$ is continuous on the compact set $\{\theta_{1,0}\}\times\Tilde{\Theta}_0(\theta_{1,0})$ and the size of $V_1\times V_2$ can be taken as small as needed, we can shrink $V_1\times V_2$ such that $\inf_{\theta\in V_1\times V_2} \ksd(\theta)>m^*(\theta_{1,n})-\epsilon.$ However, note that $\theta_{1,n}$ will finally enter $V_1$, and from then on, $\{\theta_{1,n}\}\times\Tilde{\Theta}_0(\theta_{1,n})\subset V_2$, which implies $\inf_{V_1\times V_2}\ksd(\theta)\leq m^*(\theta_{1,n})\downarrow m^*_0$. Therefore, $m^*(\theta_{1,n})-\epsilon<\inf_{V_1\times V_2}\ksd(\theta)\leq m^*(\theta_{1,n})\downarrow m^*_0$, and thus $m^*(\theta_{1,0})=m^*_0$ as $\epsilon$ is arbitrary. Therefore, $\Theta_0$ is non-empty.

2. We now show that $\Theta_0$ is compact. 

Suppose $\{B_\alpha\}\subset\Theta$ is a collection of open balls such that $\Theta_0\subset\bigcup B_\alpha$. For each $\theta_1$, the $\theta_1$-section $\Theta^{\theta_1}_0:=\{(\theta,\theta')\in\Theta_0|\theta=\theta_1\}$ of $\Theta_0$, if non-empty, is exactly the set $\Tilde{\Theta}_0(\theta_1)$. Since $\Theta^{\theta_1}_0=\Tilde{\Theta}_0(\theta_1)$ is compact, there exists a finite subcollection of balls $B^{\theta_1}_i$, $1\leq i\leq n_{\theta_1}$ such that $\{\theta_1\}\times\Tilde{\Theta}_0(\theta_1)\subset \bigcup_i B^{\theta_1}_i $. We apply lemma \ref{KSDminSet} to each $\theta_1$, and conclude there exists open neighborhoods $V^{\theta_1}_1$ and $V^{\theta_1}_2$ such that $\{\theta_1\}\times\Tilde{\Theta}_0(\theta_1)\subset V^{\theta_1}_1\times V^{\theta_1}_2 \subset \bigcup_i B^{\theta_1}_i $. If the section $\Theta^{\theta_1}_0$ is empty, we prescribe $V^{\theta_1}_1$ and $V^{\theta_2}_2$ with arbitrary neighborhood such that $\{\theta_1\}\times\Tilde{\Theta}_0(\theta_1)\subset V^{\theta_1}_1\times V^{\theta_1}_2 $. Note that $\{V^{\theta_1}_1\}_{\theta_1\in\Theta_1}$ is a open cover of $\Theta_1$, thus there exists an open cover $V^{\theta_{1,i}}_1$ for $\theta_{1,i}$, $1\leq i\leq n$. Note that $B^{\theta_{1,i}}_j$, $1\leq i\leq n$, $1\leq j\leq n_{\theta_{1,i}}$ is a finite subcover of $\Theta_0$. Therefore, $\Theta_0$ is compact.

3. We now show that the minimizers $\hat{\theta}_n\in\widehat{\Theta}_n$ will eventually populate a compact set uniformly.

Since $\Theta_0$ is compact, we choose an open set $K\subset\Theta_2$ with compact closure such that $\Theta_0\subset\Theta_1\times K$.  Since $\ksd-m^*_0$ is continuous and positive on the compact set $\Theta_1\times\partial K$, let $\epsilon:=\inf_{\Theta_1\times\partial K}\ksd(\theta)- m^*_0>0$. By theorem 4.2, $\wksd^2_n\to \ksd$ uniformly on $\Theta_1\times \Bar{K}$, which implies for large enough $n$, $\wksd_n > m^*_0+\epsilon/2 $ uniformly on $\Theta_1\times\partial K$ and $\wksd_n < m^*_0+\epsilon/2$ uniformly on $\Theta_0$. Therefore, we can apply lemma \ref{Convex} to each $\wksd_n(\theta_1,\cdot)$, and conclude that $\widehat{\Theta}_n\subset\Theta_1\times K$ for large enough $n$.

4. We now show the theorem.

Since the global minimizers of $\wksd_n$ will eventually populate the compact set $\Theta_1\times\Bar{K}$, WLOG, we assume that $\widehat{\Theta}_n$ is set of minimizers of $\wksd_n$ over $\Theta_1\times\Bar{K}$.

For each $\delta>0$, let $\Theta_\delta:=\{\theta\in \Theta_1\times\Bar{K}:d(\theta,\Theta_0)<\delta\}$ be the closed $\delta$-neighborhood of $\Theta_0$ in $\Theta_1\times \Bar{K}$, and let $\Theta^c_\delta:=\Theta_1\times\Bar{K}-\Theta_\delta$. Note that $\Theta^c_\delta$ is compact, and $\ksd-m^*0$ is positive on $\Theta^c_\delta$, thus we let $\epsilon_\delta:=\inf_{ \Theta^c_\delta} \ksd(\theta)-m^*_0>0$. For $n$ large enough such that $\widehat{\Theta}_n\subset \Theta_1\times \Bar{K}$ and $\sup|\wksd_n-\ksd|<\epsilon_\delta/2$, thus for each $\hat{\theta}_n\in\widehat{\Theta}_n$,
 \[ 
 \begin{aligned}
\ksd(\hat{\theta}_n)-m^*_0 
< \wksd_n(\hat{\theta}_n)+\epsilon_\delta/2-m^*_0 \leq \wksd_n(\theta_0)+\epsilon_\delta/2-m^*_0 < m^*_0+\epsilon_\delta-m^*_0=\epsilon_\delta,    
 \end{aligned}
\]
which implies $\hat{\theta}_n\in\Theta_\delta$, thus proving the theorem.
\end{proof}

\subsection{Proof of theorem \ref{GoF}}\label{Proof-GoF}

Due to proposition \ref{Classical-GoF}, it suffices to show that $\sum_{l=1}^n (\hat{\lambda}'_l-\hat{\lambda}_l) Z^2_l\to 0$ in probability. Note that $\hat{\lambda}'_l$, $\hat{\lambda}_l$, $1\leq l\leq n$, are the eigenvalues of the matrices $H_0:=n^{-1} (V_{\theta_0}(x_i,x_j))_{ij }$ and $H_n:=n^{-1} (V_{\hat{\theta}_n}(x_i,x_j))_{ij}$. The Weilandt-Hoffman inequality \cite[Eq. (1.64)]{tao2023topics} states that
\[
\sum_{l=1}^n | \hat{\lambda}'_l-\hat{\lambda}_l |\leq \Vert H_0-H_n \Vert_1,
\]
where $\Vert H \Vert_1:=\tr\sqrt{H^\top H} $ represents the nuclear norm of a squared matrix $H$. We will show that $\Vert H_0-H_n\Vert_1\to 0$ almost surely, which will imply that $\sum_{l=1}^n(\hat{\lambda}'_l-\hat{\lambda}_l) Z^2_l$ converges to $0$ in probability.

The proof relies on following lemma:
\begin{lemma}\label{Nuclearbound} Suppose $\mathcal{H}$ is a Hilbert space and $x_i,y_i\in \mathcal{H}$, $1\leq i\leq n$ and $\Sigma:=(\langle x_i,y_j\rangle)_{i j}$. Then $\Vert\Sigma\Vert_1\leq \Vert \Vec{x}\Vert \cdot  \Vert \Vec{y}\Vert $. Here $\Vert \Vec{x}\Vert:=\sqrt{\sum_{i=1}^n \Vert x_i\Vert^2 }$ and $\Vert \Vec{y}\Vert$ likewise. Consequently, 
\[
\Vert (\langle x_i,x_j\rangle)_{i j}-(\langle y_i,y_j\rangle)_{i j}\Vert_1\leq  \Vert \Vec{x}-\Vec{y}\Vert \cdot ( \Vert \Vec{x}\Vert+ \Vert \Vec{y}\Vert)
\]
\end{lemma}
\begin{proof} Let $\mathcal{H}_0:=$span$ \{x_1,\dots,x_n,y_1,\dots,y_n\}$ and $m:=\dim\mathcal{H}_0$. Take an orthogonal basis $\{e\}_{i=1}^m$ of $\mathcal{H}_0$ and let $x_i=\sum_{j=1}^m a^j_i e_j$ and $y_i=\sum_{j=1}^m b^j_i e_j$ for $1\leq i\leq n$, and let $A:=(a^j_i)_{i j}$, $B:=(b^j_i)_{i j}$. Then $\Sigma=A B^\top$. Let $\Sigma=P S Q^\top$ be the singular value decomposition of $\Sigma$, then apply the Cauchy-Schwartz inequality 
\[
\begin{aligned}
\Vert \Sigma\Vert_1
= \tr(S)=\tr(P^\top A B^\top Q )\leq \Vert P^\top A\Vert_F \Vert B^\top Q\Vert_F=\Vert A\Vert_F\Vert B\Vert_F.   
\end{aligned}
\]
The first part of the lemma follows from the fact that $\Vert A\Vert_F= \Vert x\Vert^2$ and $\Vert B\Vert_F= \Vert y\Vert^2$. For the second part, just note that
\[
\begin{aligned}
\Vert (\langle x_i,x_j\rangle)_{i j}-(\langle y_i,y_j\rangle)_{i j}\Vert_1
\leq \Vert (\langle x_i,x_j-y_j\rangle)_{i j}\Vert_1+\Vert(\langle x_i-y_i,y_j\rangle)_{i j}\Vert_1
\leq \Vert \vec{x}-\Vec{y}\Vert(\Vert \Vec{x}\Vert+\Vert\vec{y}\Vert ).
\end{aligned}
\]
\end{proof}
Next we prove the theorem.

\begin{proof} Recall the construction in \S\ref{KSD-LG}, the vector $\Vec{\mathcal{A}}_p k_x:=(\mathcal{A}^1_p k_x,\dots,\mathcal{A}_p^d k_x)^\top$ is an element in $\mathcal{H}^d_k$, and
\[
V_p(x_i,x_j)=\left\langle \frac{p(x_i)}{\omega(x_i)}\Vec{\mathcal{A}}_p k_{x_i} ,\frac{p(x_j)}{\omega(x_j)}\Vec{\mathcal{A}}_p k_{x_j} \right\rangle_{\mathcal{H}^d_k}.
\]
For notational simplicity, let $\phi_\theta(x):= \frac{p(x)}{\omega(x)}\Vec{\mathcal{A}}_{p_{\theta}}k_x\in\mathcal{H}^d_k$, then note that $H_0= n^{-1}(\langle\phi_{\theta_0}(x_i),\phi_{\theta_0}(x_j)\rangle)_{i j}$ and $H_n= n^{-1}(\langle\phi_{\hat{\theta}_n}(x_i),\phi_{\hat{\theta}_n}(x_j)\rangle)_{i j}$. Apply the lemma \ref{Nuclearbound} to get
\[
\begin{aligned}
n\Vert H_n-H_0\Vert_1
\leq \left(\sqrt{\sum_{i=1}^n \Vert \phi_{\hat{\theta}_n}(x_i)\Vert^2}+ \sqrt{\sum_{i=1}^n \Vert \phi_{\theta_0}(x_i)\Vert^2 }\right)\cdot \sqrt{ \sum_{i=1}^n\Vert\phi_{\hat{\theta}_n}(x_i)-\phi_{\theta_0}(x_i)\Vert^2 }.
\end{aligned}
\]
Let $f(\theta,x):=\Vert \phi_\theta(x)-\phi_{\theta_0}(x) \Vert^2$, then we can rewrite above inequality as
\begin{equation*}
\begin{aligned}
\Vert H_n-H_0\Vert_1&\leq \Bigg( \sqrt{n^{-1}\sum_{i=1}^n \Vert \phi_{\hat{\theta}_n}(x_i)\Vert^2 }+\sqrt{n^{-1}\sum_{i=1}^n \Vert \phi_{\theta_0}(x_i)\Vert^2} \Bigg)\cdot\sqrt{ n^{-1} \sum_{i=1}^n f(\hat{\theta}_n,x_i) }\\
&\leq \Bigg( \sqrt{ n^{-1} \sum_{i=1}^n f(\hat{\theta}_n,x_i) }+ 2\sqrt{n^{-1}\sum_{i=1}^n \Vert \phi_{\theta_0}(x_i)\Vert^2 } \Bigg)\cdot\sqrt{ n^{-1} \sum_{i=1}^n f(\hat{\theta}_n,x_i) }.
\end{aligned}
\end{equation*}
Clearly, by law of large numbers, 
\[
\begin{aligned}
n^{-1}\sum_{i=1}^n \Vert \phi_{\theta_0}(x_i)\Vert^2= n^{-1}\sum_{i=1}^n V_{\theta_0}(x_i,x_i) \xrightarrow{\omega-\text{a.s}}\int_G V_{\theta_0}(x,x) p_{\theta_0}(x) \mu(dx)<+\infty, 
\end{aligned}
\]
thus $n^{-1}\sum_{i=1}^n \Vert \phi_{\theta_0}(x_i)\Vert^2=O_p(1)$. 

To show $\Vert H_n-H_1\Vert \to 0$ $\omega$-a.s. (almost surely w.r.t. the probability measure $\omega$), it suffices to show $n^{-1}\sum_{i=1}^n f(\hat{\theta}_n,x_i)\to 0$ $\omega$-a.s.. Note that 
\[
f(\theta,x)\leq 2\Vert\phi_\theta(x)\Vert^2+ 2\Vert\phi_{\theta_0}(x)\Vert^2= 2V_{\theta}(x,x)+2V_{\theta_0}(x,x),
\]
thus $\sup_{\theta\in K} f(\theta,x)$ is $\omega$-integrable for any compact set $K\subset\Theta$, as the conditions in theorem \ref{MKSDEconsistency} state. Therefore, we can apply lemma \ref{Uniform-LLN} to conclude that
\begin{equation*}
n^{-1}\sum_{i=1}^n f(\theta,x_i)\to F(\theta):=\int_G f(\theta,x) \omega(d x),
\end{equation*}
compactly $\omega$-a.s. and $F(\theta)$ is continuous. Note that $f(\theta_0,x)=\Vert\phi_{\theta_0}(x)-\phi_{\theta_0}(x)\Vert^2=0$, thus $F(\theta_0)=0$. Note that $\hat{\theta}_n\to\theta_0$ $\omega$-a.s., thus $n^{-1}\sum_{i=1}^n f(\hat{\theta}_n,x_i)\to 0$ $\omega$-a.s..
\end{proof}

\bibliography{References.bib}
\bibliographystyle{IEEEtran}

%\newpage

%\input{Rebuttals/SecondRebuttal}
 
\end{document}